\documentclass{amsart}
\usepackage{graphicx,amsmath,amssymb,amsthm} 
\usepackage[colorlinks,
linkcolor=blue,
anchorcolor=blue,
citecolor=blue]{hyperref}
\usepackage{geometry}
\usepackage{verbatim}
\usepackage{enumitem}
\usepackage{natbib}
\usepackage{xcolor}
\usepackage{tikz}
\usepackage{caption}
\usepackage{subcaption}

\usepackage{mathtools}
\usepackage[final]{showlabels}

\DeclareMathOperator*{\vol}{vol}

\newcommand{\dotp}[2]{\left\langle #1, #2\right\rangle}
\newcommand{\wh}{\widehat}

\newcommand{\tr}{\mathrm{tr}}
\newcommand{\eps}{\varepsilon}
\newcommand{\mb}{\mathbb}
\def\l{\left}
\def\r{\right}
\newcommand\m{\mathcal}
\newcommand{\EHS}[3]{P^X_{#1}H\l(#2,#3\r)} 
\newcommand{\PHS}[2]{PH\l(#1,#2\r)} 
\newcommand{\HS}[2]{H\l(#1,#2\r)} 
\newcommand{\ED}[2]{D_{#1}\l(#2\r)} 

\newtheorem{theorem}{Theorem}[section]
\newtheorem{corollary}{Corollary}[section]
\newtheorem{lemma}{Lemma}[section]
\newtheorem{remark}{Remark}[section]

\title{Improved performance guarantees for Tukey's median}
\author{Stanislav Minsker}
\address{Department of Mathematics, University of Southern California}
\email{minsker@usc.edu}
\thanks{S. Minsker and Y. Shen acknowledge support by the National Science Foundation grant DMS CAREER-2045068.}
\author{Yinan Shen}
\address{Department of Mathematics, University of Southern California}
\email{yinanshe@usc.edu}

\begin{document}

\mathtoolsset{showonlyrefs=true}

\begin{abstract}
Is there a natural way to order data in dimension greater than one? The approach based on the notion of data depth, often associated with John Tukey, is among the most popular. Tukey's depth has found applications in robust statistics, graph theory, and the study of elections and social choice. 
We present improved performance guarantees for empirical Tukey's median, a deepest point associated with a given sample, when the data-generating distribution is elliptically symmetric and possibly anisotropic. Some of our results remain valid in the wider class of affine equivariant estimators. 
As a corollary of our bounds, we show that the typical diameter of the set of all empirical Tukey's medians scales like $o(n^{-1/2})$ where $n$ is the sample size. Moreover, when the data follow the bivariate normal distribution, we prove that with high probability, the diameter is of order $O(n^{-3/4}\log^{1/2}(n))$. 
On the technical side, we show how affine equivariance can be leveraged to improve concentration bounds; moreover, we develop sharp strong approximation results for empirical processes indexed by halfspaces that could be of independent interest.
\end{abstract}
\maketitle

\section{Introduction}

The fundamental notions of order statistics and ranks form a basis of many inferential procedures. They are the backbone of methods in robust statistics, a field that studies stability of algorithms under perturbations of data \citep{robuststat}. As the Euclidean space of dimension 2 and higher lacks the canonical order, the ``canonical'' versions of order statistics and ranks are difficult to construct. 
Many useful versions have been proposed over the years: examples include the componentwise ranks \citep{hodges1955bivariate}, spatial ranks and quantiles \citep{brown1983statistical,koltchinskii1997m}, Mahalanobis ranks \citep{hallin2002optimal}, and recently introduced Monge-Kantorovich ranks and quantiles \citep{chernozhukov2017monge}. 
In this paper, we will focus on the notion of \emph{data depth} that gives rise to the depth-based ranks and quantiles, see \cite{serfling2002quantile} for an overview. The idea of depth and the associated ordering with respect to a probability measure $P$ on $\mb R$ goes back to \cite{hotelling1929stability}: given $z\in \mb R$, its \emph{depth} is defined as the minimum among $P(-\infty,z]$ and $P[z,\infty)$. \cite{hodges1955bivariate} and \cite{Tukey1975MathematicsAT} extended this idea to $\mb R^2$, and finally 
\cite{donoho1982breakdown}, \cite{donoho1992breakdown} formalized the general notion of depth of a point $z\in\mb R^d$ with respect to a probability measure $P$ as the infimum of all univariate depth evaluated over projections of $P$ on the lines passing through $z$. Equivalently, Tukey's (or halfspace) depth of $z$ is 
\begin{equation}
    \label{eq:depth}
D_P(z)=\inf_{u\in S^{d-1}} P\l(H(z,u)\r),
\end{equation}
where $\HS{z}{u}=\{x\in \mathbb{R}^d:\ x^T u \geq z^T u\}$  is the halfspace passing through $z$ in direction $u$, and $S^{d-1}$ is the unit sphere in $\mb R^d$ with respect to the Euclidean norm $\|\cdot\|_2$. 
Let us remark that the notion of halfspace depth and its applications appear in other areas, for example in  graph theory \citep{small1997multidimensional,cerdeira2021centrality}, in convex geometry in relation to the convex floating bodies \citep{nagy2019data}, and the theory of social choice in economics under the name of the ``min-max majority'' \citep{caplin198864,nehring2023multi}. We refer the reader to the habilitation thesis of \cite{nagy2022halfspace} for an excellent overview of the history and recent developments in the theoretical and algorithmic aspects of Tukey's depth.

If $P$ is such that $P(\partial H) = 0$ for the boundary $\partial H$ of any closed halfspace $H$, then the infumum in \eqref{eq:depth} is attained \citep[][Proposition 4.5]{masse2004asymptotics}. For example, this is the case for distributions that are absolutely continuous with respect to the Lebesgue measure. 
Of particular interest are the points of maximal depth, that is, 
\[
\mu_\ast:=\mu_\ast(P)\in \arg\max D_P(z). 
\]
The set of deepest points with respect to $P$ will be referred to as Tukey's median set, or simply Tukey's set, while the barycenter of this set is Tukey's median; when $d=1$, it coincides with the standard median. According to Propositions 7 and 9 in \cite{rousseeuw1999depth}, Tukey's median always exists, and $\alpha_\ast:=D_P(\mu_\ast)\geq \frac{1}{d+1}$ for any $P$. 
For distributions possessing symmetry properties, $\alpha_\ast$ can be much larger. For example, if $P$ is halfspace-symmetric\footnote{The distribution of a random vector $X$ is halfspace-symmetric with respect to some $\mu\in \mb R^d$ if $P(H(\mu,u))\geq \frac12$ for all $u\in S^{d-1}$ \citep{zuo2000general}.} and absolutely continuous with respect to the Lebesgue measure, then its center of symmetry coincides with Tukey's median, and its depth $\alpha_\ast$ equals $\frac12$.   
The multivariate normal distribution, and more generally all elliptically symmetric distributions, are known to satisfy this property. Whenever $\alpha_\ast$ is large, Tukey's median also has a high breakdown point, and therefore features strong robustness properties -- for instance, its breakdown point equals $\frac13$ for halfspace-symmetric distributions, see Proposition 1 in \cite{chen1995robustness} for the precise statement. The upper-level sets of halfspace depth, defined as
\[
R_P(\alpha):=\{z\in \mb R^d: \ D_P(z)\geq \alpha_\ast - \alpha\}, \ \alpha < \alpha_\ast,
\]
are convex and compact. These sets are often called the \emph{central regions} \citep{serfling2002quantile} and their boundaries -- the quantile surfaces or depth contours \citep{liu1999multivariate}. 

Let $X_1,\ldots,X_n\in \mathbb R^d$ be a sequence of independent copies of a random vector $X$ with distribution $P$. The empirical measure $P^X_n$ is a discrete measure with atoms $X_1,\ldots,X_n$ having weight $1/n$ each. In particular, for any Borel measurable set $A\subseteq \mb R^d$, 
\[
P^X_n(A):=\frac{1}{n}\sum_{j=1}^n I_A(X_j), \text{ where } I_A(z) = \begin{cases}
    1, & z\in A, \\
    0, & z\notin A.
\end{cases}
\]
The \emph{empirical depth} corresponding to $P^X_n$ will be denoted $D_n(z)$, and it yields a natural way to order points in a sample, giving rise to the depth-based ranks. It also admits the following convenient geometric interpretation: $n D_n(z)$ equals the cardinality of a smallest subset 
$S\subset\{X_1,\ldots,X_n\}$ such that $z$ is not in the convex hull of $\{X_1,\ldots,X_n\} \setminus S$, with the convention that the cardinality of the empty set is $0$.
The empirical version of Tukey's set is defined as the set  
\[
\arg\max_{z\in \mathbb R^d}\ED{n}{z}
\]
of points with maximal empirical depth. 
The barycenter of this set is known as the empirical Tukey's median, and it will be denoted $\wh\mu_n$. The definition is illustrated in Figure \ref{fig:empiricalmedian} that displays the (empirical) depth contours, upper-level sets (central regions) and the Tukey's median  corresponding to a sample of size $n=100$ from the isotropic bivariate Gaussian distribution as well as the isotropic bivariate Student's t distribution with $2.1$ degrees freedom. 
\begin{figure}[ht]
    \centering
    \begin{subfigure}[b]{0.45\textwidth}
        \centering
        \includegraphics[width=\textwidth]{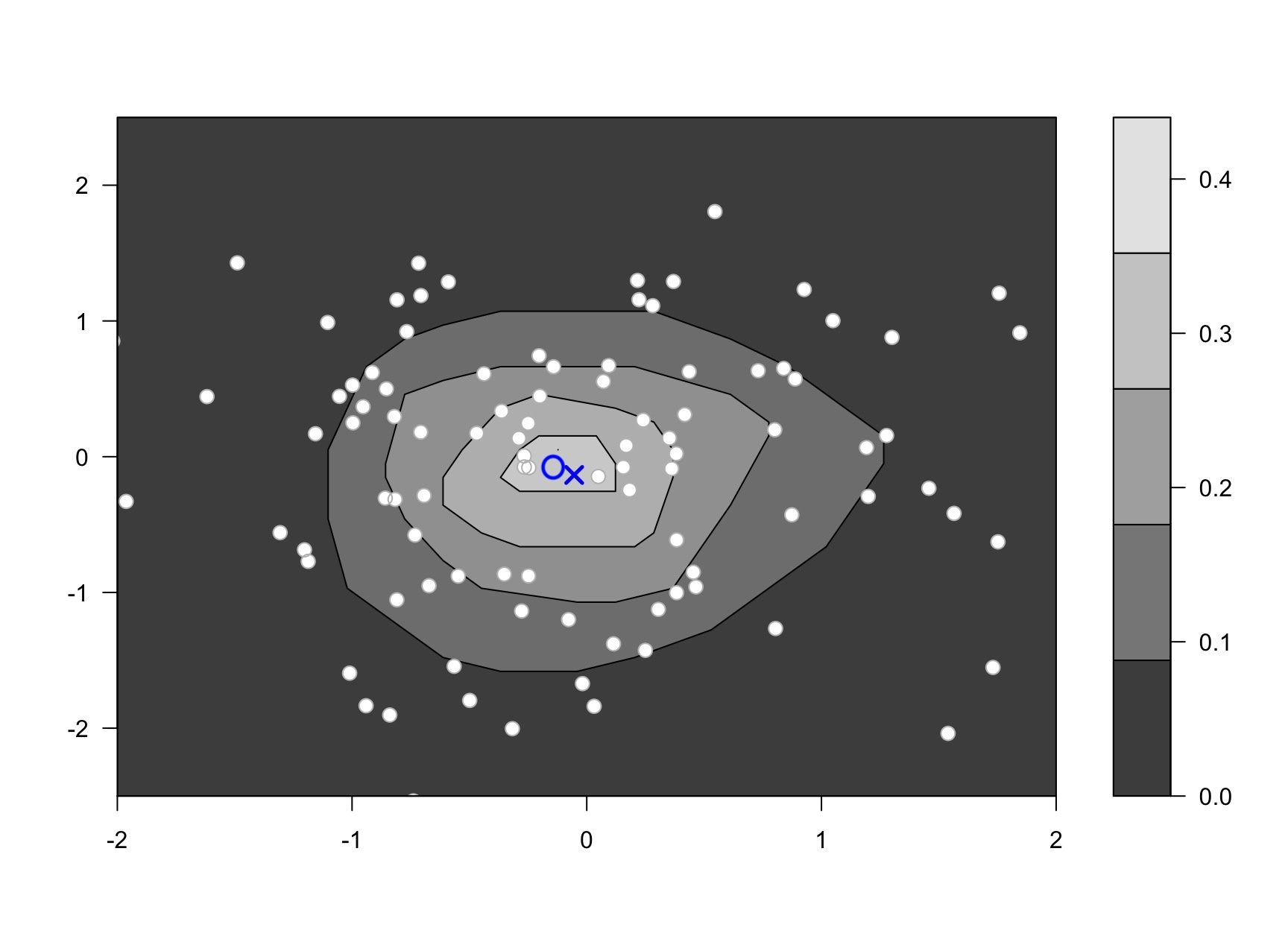}
        \caption{$n=100$, standard Gaussian}
    \end{subfigure}
    \hfill
    \begin{subfigure}[b]{0.45\textwidth}
        \centering
        \includegraphics[width=\textwidth]{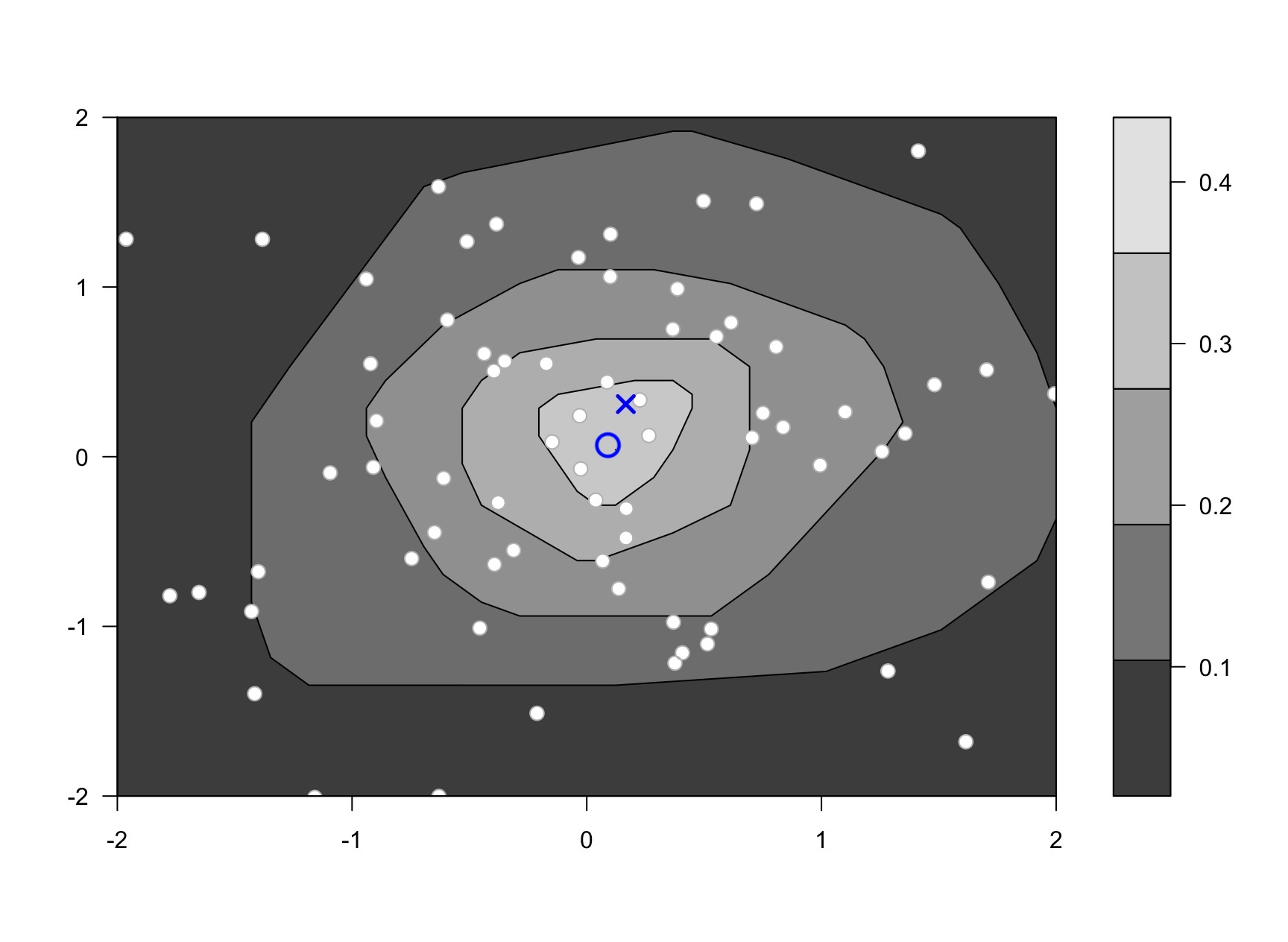}
        \caption{$n=100$, Student's t with $\nu=2.1$ d.f.}
    \end{subfigure}
    \caption{Depth contours: lighter-colored regions correspond to higher depth. Empirical Tukey's median is marked with a circle and the sample mean -- with a cross}. 
    \label{fig:empiricalmedian}
\end{figure}

\begin{figure}[ht]
    \centering
    \begin{subfigure}[b]{0.45\textwidth}
        \centering
        \includegraphics[width=\textwidth]{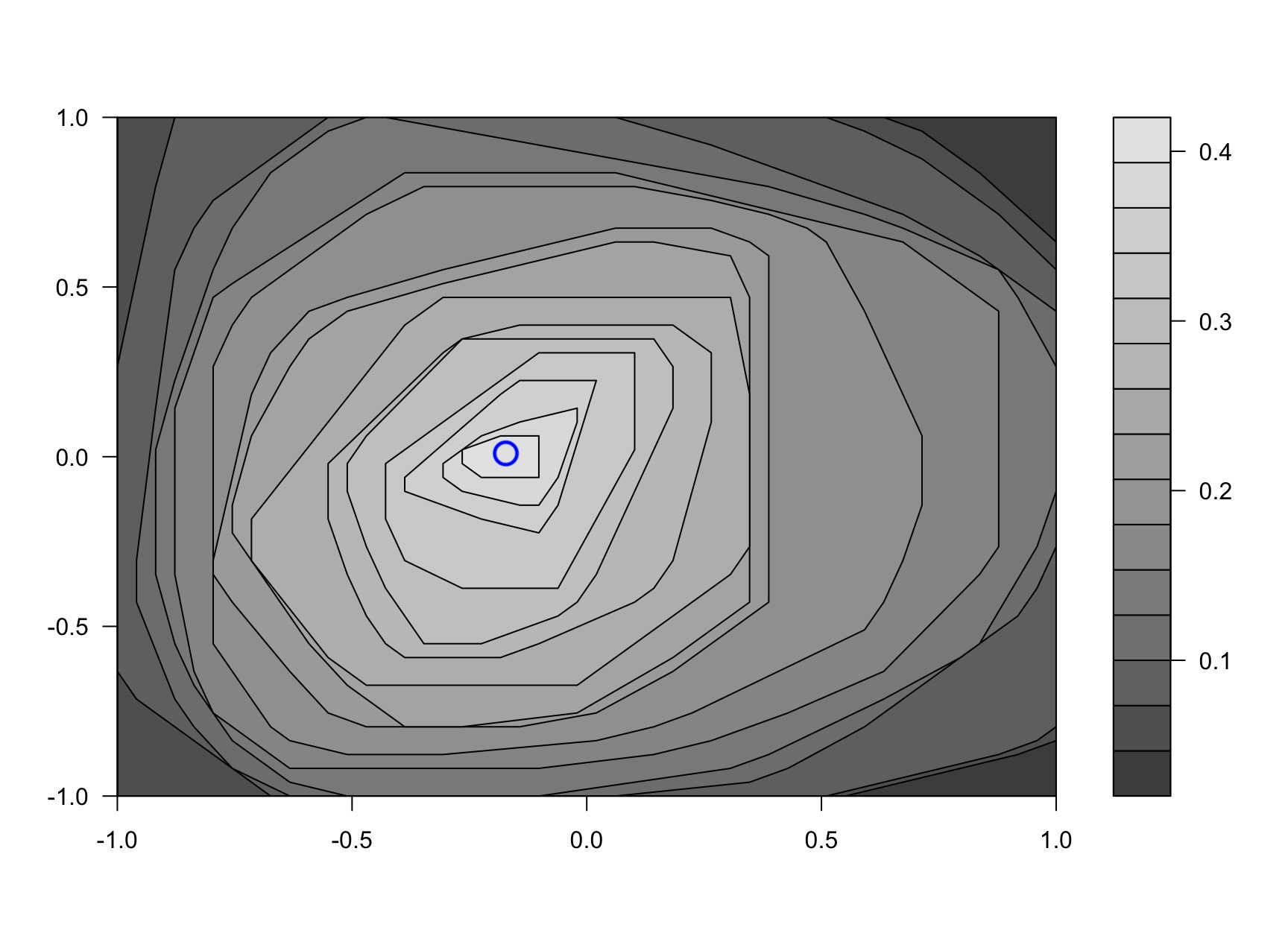}
        \caption{$n=50$, standard Gaussian}
    \end{subfigure}
    \hfill
    \begin{subfigure}[b]{0.45\textwidth}
        \centering
        \includegraphics[width=\textwidth]{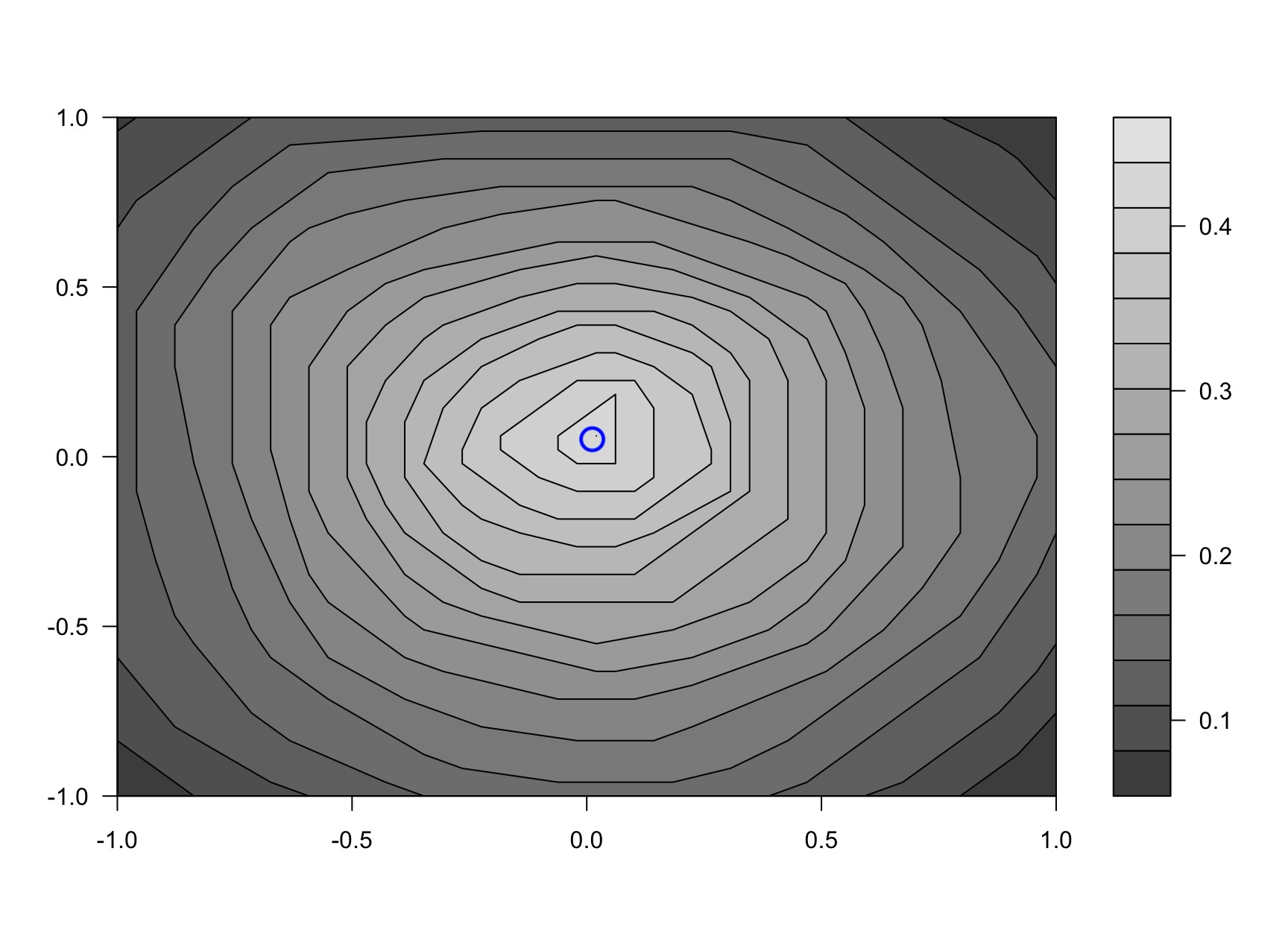}
        \caption{$n=500$, standard Gaussian}
    \end{subfigure}
    \caption{Empirical depth contours approach their population limit. The blue circle denotes Tukey's median.}
    \label{fig:contourslimit}
\end{figure}
\noindent Figure \ref{fig:contourslimit} shows that the empirical contours become smoother and approach their population counterparts as the sample size grows. Even when $\mu_\ast$ is unique, the Tukey's set can have a non-empty interior. For instance, if $d=3$, then the set $\arg\max_{z\in \mathbb R^d}\ED{n}{z}$ can be a single point, an interval, or a convex polytope, see Example 3.1 in \cite{pokorny2024another}. 
This raises a natural question: \textit{what is the typical diameter of this set?} 
In the univariate case ($d=1$), the sample median is unique when $n$ is odd, and when $n$ is even, the diameter equals $L_n = \l|X_{\l(\frac{n}{2}+1 \r)} - X_{\l(\frac{n}{2}\r)}\r|$. Under mild assumptions and using the classical results on order statistics \citep{renyi1953theory}, the expected value of $L_n$ can be shown to be of order $c/n$ for a constant $c$ that depends on the distribution. To the best of our knowledge, the question was open for $d\geq 2$. 
We make a step towards answering it and prove (see Corollary \ref{corollary:tukey-set}) that with high probability, diameter of empirical Tukey's set decays faster than $n^{-1/2}$. When $d=2$, we obtain sharper estimates and show that the diameter is of order $n^{-3/4}\log^{1/2}(n)$. These results follow from the general upper (Theorem \ref{th:diameter}) and lower (Theorem \ref{th:depth:difference}) bounds for the diameter of the sets of ``deep points'' that are proven to be sharp in some cases.  

Among the attractive properties of the halfspace depth is the \emph{affine invariance}: for any affine map $z\mapsto T(z) = Mz+b$ such that $M\in \mb R^{d\times d}$ is non-singular and $b\in \mb R^d$, 
\[
D_{P\circ T^{-1}}(Mz + b) = D_P(z), \ \forall z\in \mb R^d,
\]
where $P\circ T^{-1}(A) = P\l( T^{-1}(A)\r)$. This implies that 
\[
\mu_\ast(P\circ T^{-1}) = M \mu_\ast(P) + b,
\]
so that Tukey's median ``respects'' affine transformations.
For example, if $X$ has mean $\mu$ and the covariance matrix $\Sigma = \mb E(X-\mu)(X-\mu)^T$ is non-degenerate, then 
\[
\wh\mu_n(X_1,\ldots,X_n) = \Sigma^{1/2}\wh\mu_n(Z_1,\ldots,Z_n) + \mu, 
\]
where $Z_j = \Sigma^{-1/2}(X_j - \mu)$ has centered isotropic distribution (i.e. the covariance of $Z_1$ is the identity matrix). Due to this property, existing literature that investigates the proximity between the population median $\mu$ and its data-dependent counterpart $\wh\mu_n$ often focuses on the estimation error measured in the Mahalanobis distance $\|\Sigma^{-1/2}(\wh\mu_n(X_1,\ldots,X_n) - \mu)\|_2$. Since $\Sigma$ is typically unknown, one may instead prefer to measure the error with respect to the usual Euclidean distance $\|\wh \mu_n - \mu\|_2$. 

In this paper, we make a step towards establishing optimality properties of Tukey's median. Specifically, we prove (see Theorem \ref{th:affine}) that the size of the error $\|\wh \mu_n - \mu\|_2$ is controlled by the \emph{effective rank} of $\Sigma$ defined as the ratio of the trace $\tr(\Sigma)$ and the spectral norm $\|\Sigma\|$, 
\[
r(\Sigma) = \frac{\tr(\Sigma)}{\|\Sigma\|},
\]
as opposed to the ambient dimension $d$. The latter can be much larger than $r(\Sigma)$ if the underlying distribution is concentrated near a low-dimensional subspace of $\mb R^d$. 
We then investigate robustness properties of Tukey's median and the effects of adversarial contamination on the error. While this question has been answered for isotropic distributions (equivalently, for the error with respect to the Mahalanobis norm, see \cite{chen2018robust}), the case of Euclidean distance poses new challenges. We give a partial answer to this question in Corollary \ref{corollary:rates}. The remaining open problems are discussed in Section \ref{section:discussion}.

\subsection{Notation}
\label{section:notation}

Given $z_1,z_2\in \mb R^d$, $\dotp{z_1}{z_2}$ will stand for the standard inner product, $\|\cdot\|_2$ -- for the associated Euclidean norm and $B(z,r)$ will denote the Euclidean ball of radius $r$ centered at $z$. For a matrix $A\in \mb R^{d\times d}$, $\tr(A)$ represents its trace and $\|A\|$ - its spectral (operator) norm. $I_d$ will denote the $d\times d$ identity matrix.

We will employ the standard small-o and big-O asymptotic notation throughout the paper. Given two sequences $\{a_j\}_{j\geq 1}, \ \{b_j\}_{j\geq 1}\subset \mb R_+$, we will say that $a_j \ll b_j$ if $a_j = o(b_j)$ as $j\to\infty$. 
We will write $o_P(1)$ to denote a sequence of random variables $\{\xi_j, \ j\geq 1\}$ that converge to $0$ in probability as $j\to\infty$ and $O_P(1)$ -- a sequence of random variables that is stochastically bounded. Indicator function of an event $\m E$ will be denoted $I_{\m E}$ or $I\{\m E\}$. Finally, $\Phi(t)$ will stand for the cumulative distribution function of the standard normal law on $\mb R$ and $\phi(t)$ -- for the associated probability density function. Additional notation will be introduced on demand.

\section{Main results}
\label{section:main}

Assume that a fraction of the data has been replaced by arbitrary values, that is, we observe $Y_1,\ldots,Y_n$ where $\frac{1}{n}\sum_{j=1}^n I\{Y_j\ne X_j\} = \eps$.\footnote{This is commonly referred to as the \emph{adversarial contamination framework}, see \cite{diakonikolas2023algorithmic}} 
In a scenario when $X_1,\ldots,X_n$ have multivariate normal distribution $N(\mu,\Sigma)$ such that $\Sigma$ is non-singular, \citet[][Theorem 2.1]{chen2018robust} showed that whenever $\eps<1/5$ and $\frac{d}{n}+\frac{t}{n}\leq c$, the inequalities
\begin{align}
\label{eq:bound:isotropic}
\|\Sigma^{-1/2}\l(\wh\mu_n(Y_1,\ldots,Y_n) - \mu \r)\|_2 &\leq C\l( \sqrt{\frac{d+t}{n}} + \eps \r) \text{ and } 
\\
\sup_{z\in \mb R^d}\l| {D_P(z) - D^Y_n(z)} \r|\ &\leq C_1\l( \sqrt{\frac{d+t}{n}} + \eps \r)
\end{align}
hold with probability at least $1-e^{-t}$, where $c,C,C_1>0$ are absolute constants and $D^Y_{n}(z) := \inf_{\|u\|_2=1} P^Y_{n} H(z,u)$. The key feature of these inequalities is the fact that it implies optimal dependence of the error $\|\Sigma^{-1/2}\l(\wh\mu_n(Y_1,\ldots,Y_n) - \mu \r)\|_2$ on the contamination proportion $\eps$; Tukey's median is among the first estimators known to possess this property. This robustness property of Tukey's median is illustrated in Figure \ref{fig:robustness} where it is juxtaposed with the sample mean. 

\begin{figure}[t]
    \centering
    \begin{subfigure}[b]{0.45\textwidth}
        \centering
        \includegraphics[width=\textwidth]{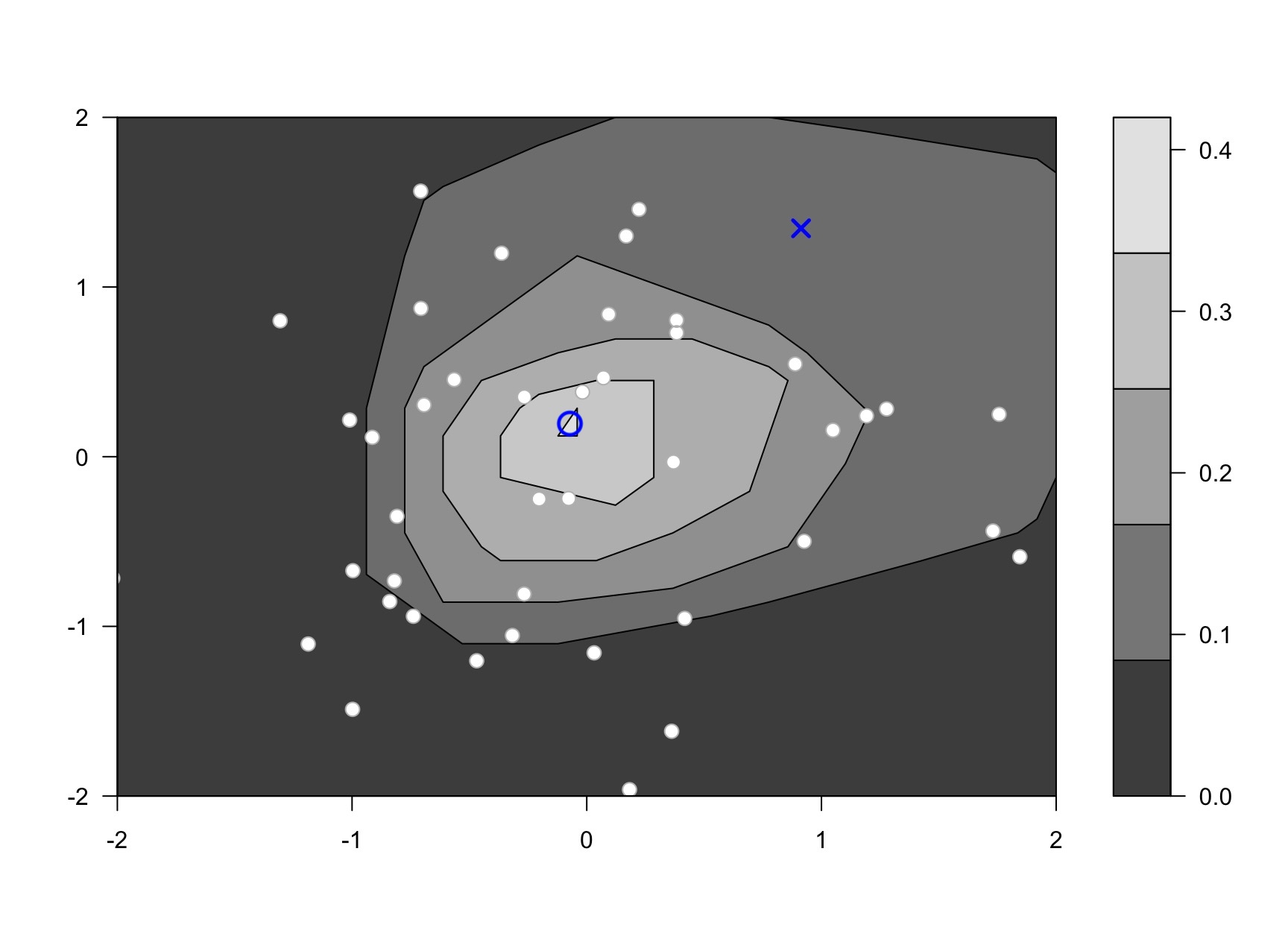}
        \caption{$n=50$, standard Gaussian}
    \end{subfigure}
    \hfill
    \begin{subfigure}[b]{0.45\textwidth}
        \centering
        \includegraphics[width=\textwidth]{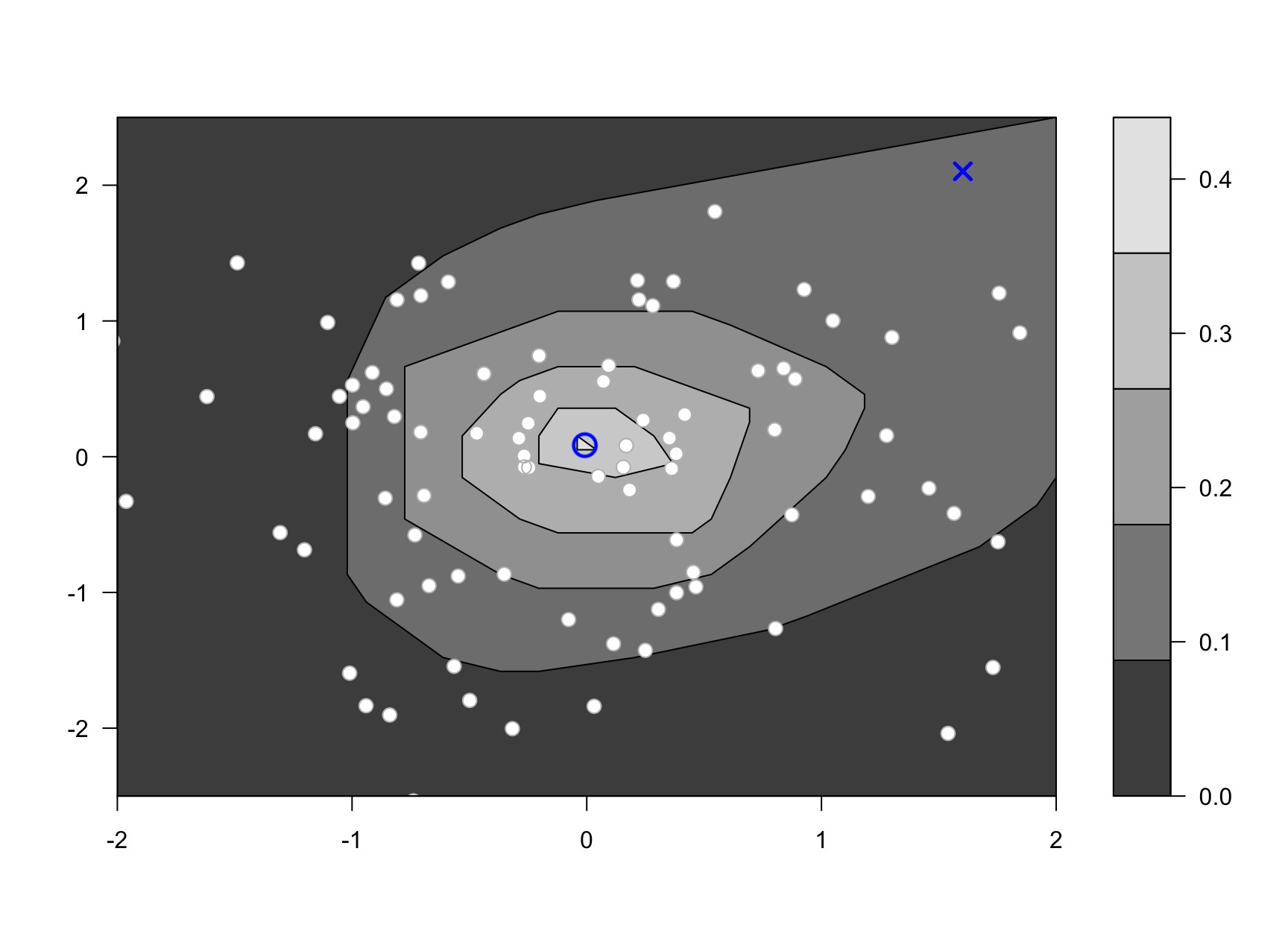}
        \caption{$n=100$, standard Gaussian}
    \end{subfigure}
    \\
    \begin{subfigure}[b]{0.45\textwidth}
        \centering
        \includegraphics[width=\textwidth]{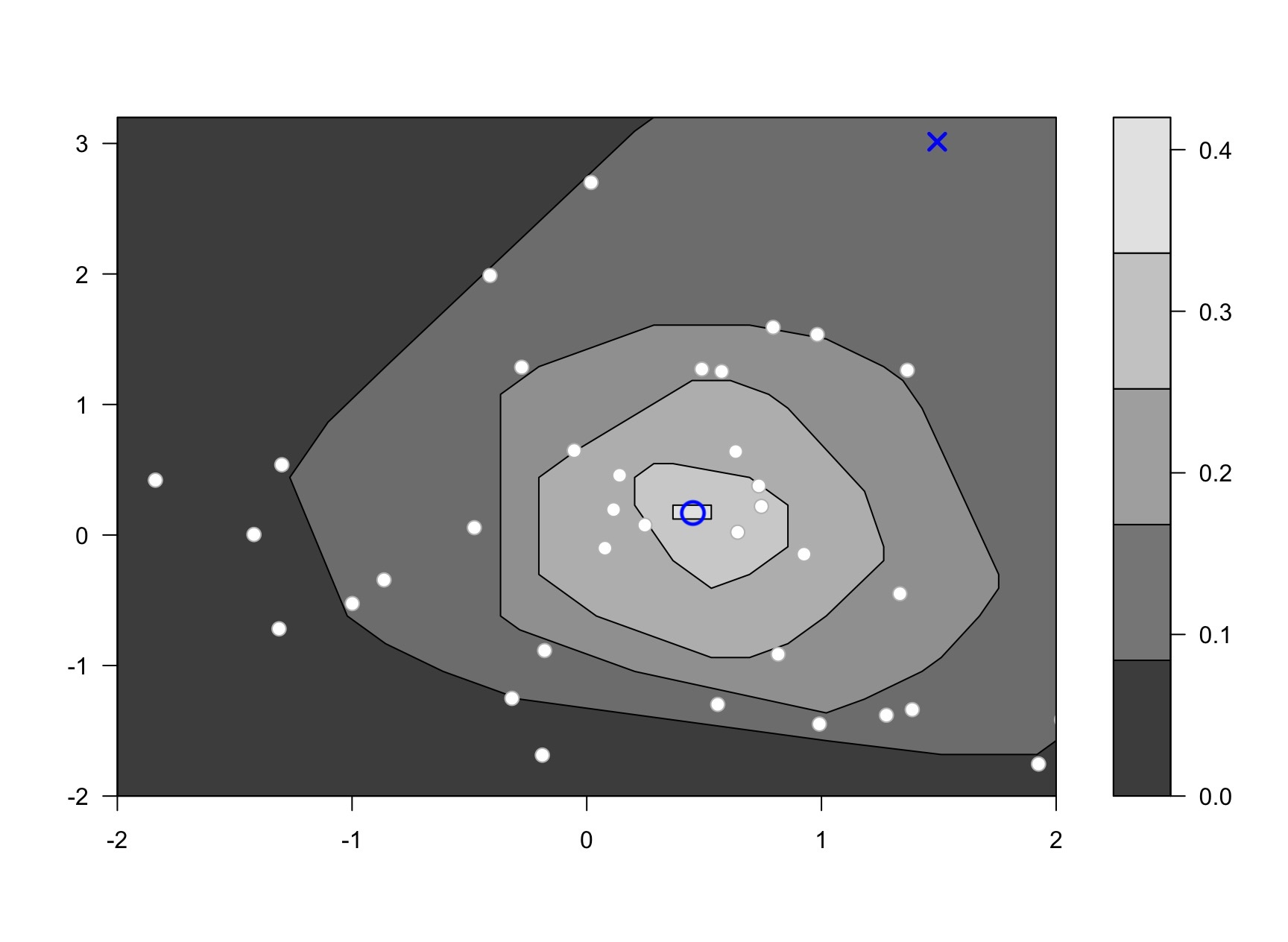}
        \caption{$n=50$, Student's t with $\nu=2.1$ d.f.}
    \end{subfigure}
    \hfill
    \begin{subfigure}[b]{0.45\textwidth}
        \centering
        \includegraphics[width=\textwidth]{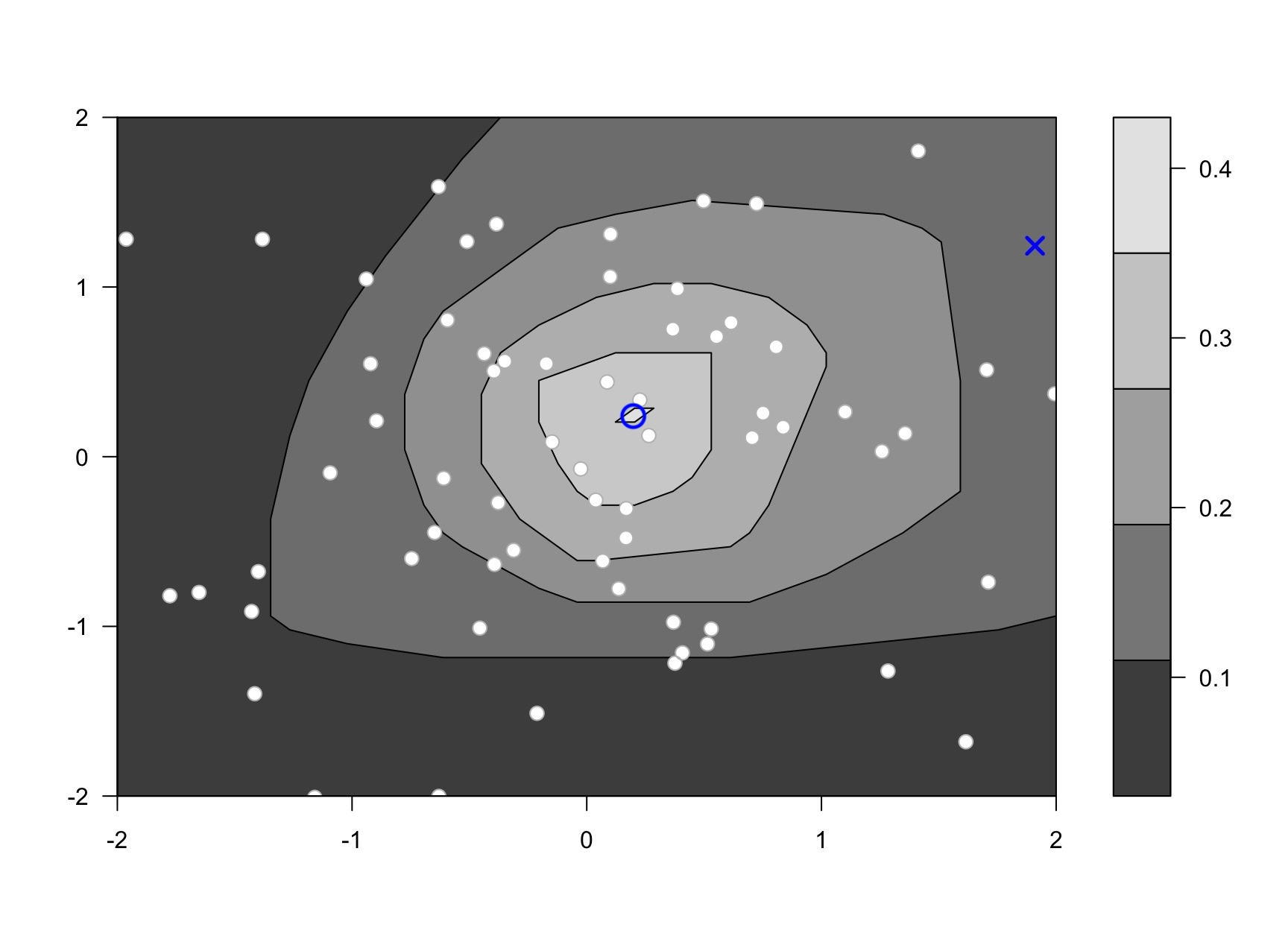}
        \caption{$n=100$, Student t with $\nu=2.1$ d.f.}
    \end{subfigure}
    \caption{Contaminated sample with contamination proportion $\eps=0.1$. Circle mark denotes Tukey median and the cross represents the mean.}
    \label{fig:robustness}
\end{figure}
More recently, \cite{minasyan2023statistically} proved that there exists an estimator $\tilde \mu_n$ that satisfies the inequality
\begin{equation}
\label{bound:optimal}
\|\tilde\mu_n(Y_1,\ldots,Y_n) - \mu \|_2 \leq C\|\Sigma\|^{1/2}\l( \sqrt{\frac{r(\Sigma)+t}{n}} + \eps \r)
\end{equation}
with probability at least $1-e^{-t}$ whenever $\frac{r(\Sigma)+t}{n}\leq c$.   Furthermore, if $\|\Sigma\|$ and $r(\Sigma)$ are known up to some absolute multiplicative factor, then the requirement $\frac{r(\Sigma)+t}{n}\leq c$ can be dropped. This inequality exhibits optimal dependence on the ``dimensional'' characteristic expressed via the effective rank $r(\Sigma)$ as opposed to the ambient dimension $d$, as well as correct dependence on the degree of contamination $\eps$.  

Deviation inequalities similar to bound \eqref{bound:optimal} are often referred to as \emph{sub-Gaussian guarantees}, since they hold, for $\eps=0$, for the sample mean of i.i.d. Gaussian random vectors due to Borell-TIS inequality \citep{borell1975brunn,cirel1976norms}. 
Whether Tukey's median satisfies an estimate of the form \eqref{bound:optimal} remained an open question: \emph{specifically, it was unknown whether $d$ can be replaced by $r(\Sigma)$, even for $\eps = 0$}. We will give a (partial) affirmative answer below: namely, we show that (a) inequality \eqref{bound:optimal} does hold for empirical Tukey's median when $\eps = 0$ (Corollary \ref{th:main}), and that (b) for large $n$, the term $\|\Sigma\|^{1/2}\l( \sqrt{\frac{r(\Sigma)+t}{n}}\r)$ dominates the error bound when $\eps\leq c(d)n^{-1/2}$ for a sufficiently small constant $c(d)$, see Corollary \ref{corollary:rates}. 
As the question we are investigating is mainly interesting in the situation when $d$ is large, we will assume that $d\geq 2$ in the remainder of the paper. 


\subsection{Contamination-free framework}

Our starting point is the contamination-free scenario corresponding to $\eps = 0$. 
Recall that a random vector $X\in\mathbb{R}^d$ has elliptically symmetric distribution $\m E(\mu,\Sigma,F)$, denoted $X\sim \m E(\mu,\Sigma,F)$ in the sequel, if
\begin{equation}
\label{def:elliptical}
X\stackrel{d}{=}\eta \cdot B U + \mu,
\end{equation}
where $\stackrel{d}{=}$ denotes equality in distribution, $\eta$ is a nonnegative random variable with cumulative distribution function $F$, $B$ is a fixed $d\times d$ matrix such that $\Sigma = BB^T$, and $U\in \mb R^d$ is a vector uniformly distributed over the unit sphere $\mathbb{S}^{d-1}$ and independent of $\eta$. 
The distribution $\mathcal{E}(\mu,\Sigma,~F)$ is well defined, as if $B_1 B_1^T=B_2 B_2^T$, then there exists an orthogonal matrix $Q$ such that $B_1=B_2 Q$, and $QU\stackrel{d}{=}U$. 
To avoid such ambiguity, we will allow $B$ to be any matrix satisfying $BB^T=\Sigma$. 
Whenever  $\mb E\eta^2<\infty$, we will also assume that $\eta$ is scaled so that $\mb E\eta^2 = d$, in which case the covariance matrix of $X$ exists and is equal to $\Sigma$. 

A useful property of elliptically symmetric distributions that is evident from the definition is the following: for every orthogonal matrix $Q$, $Q\Sigma^{-1/2}(X-\mu)$ has the same distribution as $\Sigma^{-1/2}(X-\mu)$. 
In other words, $\Sigma^{-1/2}(X-\mu)$ is spherically symmetric. 
We are ready to state our first result that essentially hinges on this fact. 
\begin{theorem}
\label{th:affine}
    Let $X_1,\ldots,X_n$ be i.i.d. copies of $X\sim \mathcal E(\mu,\Sigma,F)$ where $\Sigma$ is non-singular. Let $\widetilde\mu_n$ be any affine equivariant estimator of $\mu$.
    Suppose that the inequality
    \[
    \|\Sigma^{-1/2}(\widetilde\mu_n - \mu) \|_2 \leq e(n,t,d)
    \]
    holds with probability at least $1-p(t)$. Then 
    \[
    \|\widetilde\mu_n - \mu\|_2 \leq e(n,t,d)\|\Sigma\|^{1/2}\l( \sqrt{\frac{r(\Sigma)}{d}} + \sqrt{\frac{2t}{d}}\r)
    \]
    with probability at least $1-p(t) - e^{-t}$.
\end{theorem}
\noindent In typical scenarios, 
\[
e(n,t,d) \asymp \sqrt{\frac{d}{n}} + \sqrt{\frac{t}{n}} \text{ and } p(t) = e^{-t},
\]
whence 
\begin{equation}
\label{eq:combination}
   \|\widetilde\mu_n - \mu\|_2 \leq C\left( \sqrt{\frac{d}{n}} + \sqrt{\frac{t}{n}}\right)\sqrt{\|\Sigma\|}\left( \sqrt{\frac{r(\Sigma)}{d}} + \sqrt{\frac{t}{d}}\right)
    \leq C_1\sqrt{\|\Sigma\|}\left( \sqrt{\frac{r(\Sigma)}{n}}+ \sqrt{\frac{t}{n}}\right)
\end{equation}
with probability at least $1-2e^{-t}$ whenever $t\leq d$.
The proof of the Theorem \ref{th:affine} is given in Appendix \ref{section:proof:affine}. Next, we state the corollaries for Tukey's median and the Stahel-Donoho estimator, another well known affine-equivariant robust estimator of location, see  \cite{stahel1981breakdown,donoho1982breakdown} for its definition and properties. 
\begin{corollary}[Tukey's median]
\label{th:main}
Let $X_1,\ldots, X_n$ be i.i.d. copies of $X\sim N(\mu,\Sigma)$ where $\Sigma$ is non-singular, and let $\wh\mu_n$ be Tukey's median. Then there exist absolute constants $c,C>0$ such that
    \[
    \left\|\wh \mu_n - \mu \right\|_2 \leq C\sqrt{\|\Sigma\|}\left(\sqrt{\frac{r(\Sigma)}{n}} + \sqrt{\frac{t}{n}} \right)
    \]
    with probability at least $1-2e^{-t}$ whenever $\frac{d}{n}+\frac{t}{n}\leq c$. 
\end{corollary}
Compared with the previous results by \cite{chen2018robust}, the bound of Corollary \ref{th:main} demonstrates that the estimation error depends on the effective rank $r(\Sigma)$ rather than the ambient dimension $d$.
\begin{proof}
    Theorem 2.1 in \cite{chen2018robust} states that there exist an absolute constants $c,C>0$ such that
    \[
    e(n,t,d) = C\left(\sqrt{\frac{d}{n}} +\sqrt{\frac{t}{n}} \right)
    \] 
    and $p(t) = e^{-t}$ whenever $\frac{d}{n}+\frac{t}{n}\leq c$. Inequality \eqref{eq:combination} thus implies the desired result whenever $t\leq d$. If $t>d$, it suffices to show that under the stated assumptions, 
    $\|\Sigma^{-1/2}(\wh\mu_n - \mu)\|\leq C\sqrt{\frac{t}{n}}$ with probability at least $1-2e^{-t}$, which follows directly from Theorem 2.1 in \cite{chen2018robust}.
\end{proof}
\begin{remark}
    The bound of Corollary \ref{th:main} remains valid even when $d\geq n+1$: indeed, in this case with probability $1$ the convex hull of $X_1,\ldots,X_n$ is a simplex with $n$ vertices, and every point in the convex hull has (maximal) depth $1/n$, hence $\wh\mu_n$ coincides with the sample mean. This reasoning also demonstrates that generally, deviation guarantees for Tukey's median do not have to be sub-Gaussian in the regime $d\geq n+1 $ (e.g. when $X$ follows elliptically symmetric distribution where the random variable $\eta$ has heavy-tailed distribution). 
\end{remark}
\noindent Up to the values of numerical constants, the same result is valid for the Stahel-Donoho estimator. 
\begin{corollary}[Stahel-Donoho estimator]
\label{th:stahel}
Let $X_1,\ldots, X_n$ be i.i.d. copies of $X\sim N(\mu,\Sigma)$ where $\Sigma$ is non-singular, and let $\widetilde\mu_n$ be the barycenter of the set of Stahel-Donoho estimators. Then there exist absolute constants $c,C>0$ such that
    \[
    \left\|\widetilde \mu_n - \mu \right\|_2 \leq C\left(\sqrt{\frac{\tr(\Sigma)}{n}} + \sqrt{\|\Sigma\|}\sqrt{\frac{t}{n}} \right)
    \]
    with probability at least $1-2e^{-t}$ whenever $\frac{d}{n}+\frac{t}{n}\leq c$. 
\end{corollary}
\begin{proof}
    The argument repeats the proof of Corollary \ref{th:main} \emph{mutatis mutandis}, where Theorem 1 in \cite{depersin2023robustness} plays the role of Theorem 2.1 in \cite{chen2018robust}.
\end{proof}
\begin{remark}
    Results of both corollaries hold not just for Gaussian measures but more generally for a wide class of elliptically symmetric distributions. Indeed, the arguments in \cite{chen2018robust} and \cite{depersin2023robustness} leading to the bounds for $e(n,t,d)$  remain valid if we assume that $X\sim \m E(\mu,\Sigma,F)$ is such that the density $q(t)$ of one-dimensional projections
    \[
    \dotp{\Sigma^{-1/2}(X-\mu)}{v} = \eta\dotp{U}{v}, 
    \]
    where $\dotp{\cdot}{\cdot}$ denotes the standard Euclidean inner product, is continuous and strictly positive at the origin (also see Remark 2.2 in \cite{chen2018robust}). For example, multivariate t-distribution satisfies this assumption. 
\end{remark}

\subsection{Performance guarantees in the adversarial contamination framework}

Recall the adversarial contamination framework that was introduced in the beginning of Section \ref{section:main}. 
In what follows, we will let ${\wh\mu}_{n}^X$ be the empirical Tukey's median based on the uncontaminated data $X_1,\ldots,X_{n}$ and 
$\wh\mu_n^Y$ -- its counterpart based on the sample $Y_1,\ldots,Y_n$. 
It is easy to see that $\wh\mu_n^Y$ must have large empirical depth relative to the uncontaminated sample. Indeed, let
\[
D^X_{n}(z) = \inf_{\|u\|_2=1} P^X_{n} H(z,u)
\]
stand for the empirical depth of a point $z\in \mb R^d$, where the empirical distribution $P_n^X$ is based on $X_1,\ldots,X_{n}$, and 
\(
{\widehat d}_{n} := D_{n}^X\l({\wh\mu}_{n}^X\r)
\)
be the maximal (empirical) depth. Consider the set (the central region)
\[
R_{n}(\delta) := \{ z\in \mb R^d: \ D^X_{n}(z)\geq {\widehat d}_{n} - \delta \}
\]
of ${\widehat d}_{n}-\delta$-deep points, which is known to be a convex polytope whose facets belong to the hyperplanes defined by the elements of the sample. 
\begin{lemma}
\label{lem:deep set}
Let $\wh\mu_n^Y\in \arg\max_{z\in \mb R^d} D^Y_{n}(z)$. 
Then 
\[
\wh\mu^Y_n \in R_n(2\eps).
\]
\end{lemma}
\begin{proof}
Observe that the empirical depth $D^Y_n(z)$ (based on the contaminated dataset) of any point $z\notin R_{n}(\delta)$ satisfies the inequality 
\[
D^Y_n(z) < \eps + {\widehat d}_{n} - \delta.
\]
Indeed, 
\begin{equation}
\label{eq:depth-difference}
    D^Y_n(z) = \inf_{\|u\|_2=1} P^Y_{n} H(z,u)
    \leq \underbrace{\inf_{\|u\|_2=1} P^X_{n} H(z,u)}_{<{\widehat d}_{n} - \delta} + \underbrace{\sup_{\|u\|_2=1} \l| \l(P^Y_{n} - P^X_n\r) H(z,u) \r|}_{\leq \eps}.
\end{equation}
The latter does not exceed ${\widehat d}_{n} - \eps$ whenever $\delta>2\eps$. Since $D_n^Y\l({\wh\mu}_{n}\r)$ is no smaller than $\widehat d_{n}-\eps$ (this is easily checked in a way similar to inequality \eqref{eq:depth-difference}), the result follows. 
\end{proof}

Lemma~\ref{lem:deep set} allows us to reduce the question of understanding the effect of contamination to a problem that only involves the original sample $X_1,\ldots,X_n$: indeed, its immediate corollary is the inequality

\begin{equation}
    \label{eq:crude-bound}
\|\wh\mu_n^Y - \mu\|_2 \leq \|\wh\mu_n^X - \mu\|_2 + \mathrm{diam}\l(R_n(2\eps)\r),
\end{equation}
where $\mathrm{diam}(A)$ stands for the diameter of a set $A\subseteq\mb R^d$. Results by \cite{chen2018robust} (see the proof of Theorem 2.1 therein; also see \cite{brunel2019concentration} for wider classes of distributions) imply that 
\begin{equation}
    \label{eq:isotropic2}
\textrm{diam}\l(R_{n}(2\eps) \r)\leq C\sqrt{\|\Sigma\|}\l(\sqrt{\frac{d+t}{n}}+\eps\r)
\end{equation}
with probability at least $1-e^{-t}$ whenever $\frac{d+t}{n} +\eps\leq c$ for some absolute constants $c,C>0$. In particular, it trivially yields, together with \eqref{eq:crude-bound}, that for $\eps\geq \sqrt{\frac dn}$, 
\begin{equation}
\label{eq:bigepsilon}
    \left\|\wh \mu_n^Y - \mu \right\|_2 \leq C\left( \sqrt{\frac{\tr(\Sigma)}{n}} + \sqrt{\|\Sigma\|}\left(\sqrt{\frac{t}{n}} +\eps\right) \right)
\end{equation}
with probability at least $1-2e^{-t}$. Therefore, in the remainder of this section, we will be primarily interested in the situation when  
\(
\eps:=\eps_n \leq \sqrt{\frac{d}{n}}.
\)
To this end, we will assume that $d$ and $r(\Sigma)$ are fixed, where $d$ is potentially much larger than $r(\Sigma)$, and that $n\to\infty$. 
The following theorem is our main result in this direction. 
\begin{theorem}
\label{th:diameter}
Assume that $X_1,\ldots,X_n$ are independent random vectors sampled from the isotropic normal distribution $N(\mu,I_d)$. Then
\begin{enumerate}
\item[(i)] for any $\eps_n\in[0,1/2)$ and $d\geq 3$, 
    \[
    \mathrm{diam}(R_{n}(\eps_n)) = \frac{1}{\sqrt{n}}+O_P\l(\eps_n \vee n^{-\frac{1}{2(d-1)}-\frac{1}{2}}\log^{3/2}(n)\r);
    \]
\item[(ii)] for any sequence $\eps_n = o(n^{-1/2})$ and $d\geq 3$, 
    \[
    \mathrm{diam}(R_{n}(\eps_n)) = o_P\l(n^{-1/2}\r);
    \]
\item[(iii)] if $d=2$, then for $\eps_n \in[0,1/2)$, 
\[
\mathrm{diam}(R_{n}(\eps_n))  = O_P\l(\eps_n \vee \frac{\log^{1/2}(n)}{n^{3/4}}\r).
\]
\end{enumerate}
\end{theorem}
\noindent Let us remark that Corollary \ref{corollary:diameter} stated below implies that the bound \textbf{(iii)} for $d=2$ is sharp for $\eps_n\gtrsim n^{-3/4}\log^{1/2}(n)$. Next, we state the immediate implications of this theorem.
\begin{corollary}[Tukey's median set]
    \label{corollary:tukey-set}
The diameter of the set $R_n(0)$ satisfies
\[
\mathrm{diam}(R_n(0)) = o_P\l( n^{-\frac{1}{2}} \r).
\]
Moreover, when $d=2$, 
\[
\mathrm{diam}(R_{n}(0)) = O_P\l(n^{-3/4}\log^{1/2}(n)\r).
\]
\end{corollary}
Results of this type (besides the well known case $d=1$ that was mentioned in the introduction) appear to be new in the literature. It would be interesting to improve the $o_P\l(n^{-1/2}\r)$ rate for $d\geq 3$; the technical challenges that have to be solved in order to answer this question are related to the analysis of random concave functions. We elaborate on this point in more detail in Section \ref{section:discussion}. 
Next, we explain the implications of Theorem \ref{th:diameter} for the estimation error of Tukey's median in the presence of contamination. 
\begin{corollary}[Estimation error under adversarial contamination]
\label{corollary:rates}
   Assume that $d\geq 3$ and that $X_1,\ldots,X_n$ are i.i.d. random vectors sampled from the multivariate normal distribution $N(\mu,\Sigma)$. In the adversarial contamination framework, 
    \[
    \left\|\wh \mu_n^Y - \mu \right\|_2 \leq C\left( \sqrt{\frac{\tr(\Sigma)}{n}} + \sqrt{\|\Sigma\|}\sqrt{\frac{t}{n}} \right)+O_P\l(\eps_n\vee n^{-\frac{1}{2(d-1)}-\frac{1}{2}}\log^{3/2}(n)\r)
    \]
    on an event of probability at least $1-2e^{-t}$, where $C>0$ is an absolute constant.
\end{corollary}
In particular, when $n$ is large and $\eps_n \leq  c(d)n^{-1/2}$ for $c(d)$ small enough, we conclude that the estimation error is still controlled by the effective rank $r(\Sigma)$ rather than the ambient dimension $d$.  
\begin{proof}
    First, note that due to affine invariance of the depth function $D_n(z)$, the depth regions corresponding to the data $X_1,\ldots,X_n$ and $\Sigma^{-1/2}(X_1-\mu),\ldots,\Sigma^{-1/2}(X_n-\mu)$ are related via
    \[
    R_n(2\eps_n;X_1,\ldots,X_n) = \mu + \Sigma^{1/2}\,R_n(2\eps_n;\Sigma^{-1/2}(X_1-\mu),\ldots,\Sigma^{-1/2}(X_n-\mu)).
    \]
    Therefore, 
    \[
    \mathrm{diam}\l( R_n(2\eps_n;X_1,\ldots,X_n) \r) \leq \|\Sigma\|^{1/2}\mathrm{diam}\l( R_n(2\eps_n;\Sigma^{-1/2}(X_1-\mu),\ldots,\Sigma^{-1/2}(X_n-\mu)) \r).
    \]
    Inequality \eqref{eq:crude-bound} combined with part (i) of Theorem \ref{th:diameter} and Corollary \ref{th:main} yields that 
\begin{equation}
    \label{eq:sum:v1}
    \left\|\wh \mu_n^Y - \mu \right\|_2 \leq C\left( \sqrt{\frac{\tr(\Sigma)}{n}} +\sqrt{\|\Sigma\|}\sqrt{\frac{t}{n}} \right)+ \sqrt{\frac{\|\Sigma\|}{n}} + O_P\l(\eps_n \vee n^{-\frac{d}{2(d-1)}}\log^{3/2}(n)\r)
\end{equation}
on event of probability at least $1-2e^{-t}$, and the desired conclusion follows. 
\end{proof}
\noindent Next, we turn our attention to the proof of the key technical result, Theorem \ref{th:diameter}.
\subsection{Proof of Theorem \ref{th:diameter}}
\textbf{Overview of the proof.} Everywhere below, we will assume without loss of generality that $\mu=0$. 
The main idea of the proof is to show that with high probability, the trajectories of the empirical depth process $D_n(z)$ restricted to a small neighborhood of Tukey's median $\wh \mu_n$ can be uniformly approximated by (random) concave functions that have unique maximizers. In turn, this allows one to control diameters of the central regions (the upper level sets of the function $D_n(z)$), see Figure \ref{fig:concave} for the visual illustration. 
\begin{figure}[t]
      \centering
       \includegraphics[width=0.5\textwidth]{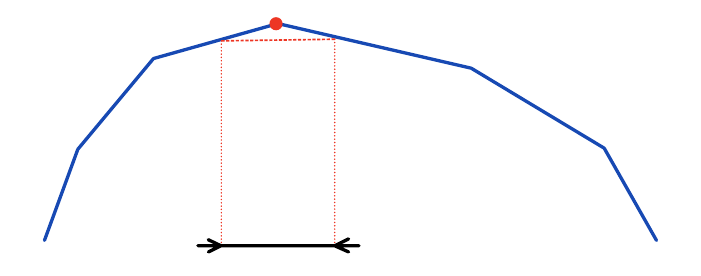}     
       \caption{A concave function with a unique maximizer, and its upper level set.}
         \label{fig:concave}
\end{figure}

These concave functions are defined as (a variant of\footnote{Typically, the Legendre-Fenchel transform is defined as the supremum of a family of linear functions, instead of the infumum.}) the Legendre-Fenchel transform of the Brownian bridge process indexed by halfspaces. The main technical challenges are (a) to establish control of the difference between the trajectories of the empirical depth process and aforementioned random concave functions, and (b) understand the behavior of the random concave functions near its maxima. The first goal is accomplish via new Koml\'{o}s–Major–Tusn\'{a}dy type strong approximation results that rely on the fundamental estimates due to \cite{koltchinskii1994komlos}, while the second problem is analyzed using geometric properties of subdifferentials of Legendre-Fenchel transforms. 

We will now present the required technical details. Everywhere below, $C>0$ will denote a sufficiently large absolute constant. 
Recall that in view of inequality \eqref{eq:isotropic2}, the upper level set $R_{n}(2\eps_n)$ of empirical Tukey's depth satisfies
    \[
    R_{n}(2\eps_n)\subseteq B\l(0,C_1\sqrt{\frac{d+t}{n}}\r)
    \]
    with probability at least $1-e^{-t}$. Next, define stochastic process $W_n(z)$ via
    \begin{equation}
    \label{def:W_n(z)}
    W_{n}(z) = \sqrt{n}\inf_{\|v\|_2=1} \l( \EHS{n}{\frac{z}{\sqrt{n}}}{v} - \frac{1}{2}\r), \ \|z\|_2\leq C_1\sqrt{d+t}.
    \end{equation}
    The process $W_{n}(z)$ is simply a rescaled version of the empirical depth process in the neighborhood of $\mu$. 
    On event of probability at least $1-e^{-t}$, we can equivalently express $R_{n}(2\eps_n)$ via
    \[
    R_{n}(2\eps_n)=\l\{\frac{z}{\sqrt n}:\ \|z\|_2\leq C\sqrt{d+t} \text{ and } W_{n}(z)\geq \sup_{z\in \mb R^d} W_{n}(z) - 2\sqrt{n}\eps_{n}\r\}.
    \]
    It is well known (for example, see page 296 in the paper by \cite{masse2002asymptotics}) that the rescaled depth process $W_{n}(z)$ converges weakly to 
    \begin{equation}
	\label{def:W(z)}
    W(z) := W_G(z) = \inf_{\|v\|_2=1} \l\{ G(v) - \frac{1}{\sqrt{2\pi}} \langle z, v\rangle\r\},\ \|z\|_2 \leq C_1\sqrt{d+t},
    \end{equation}
    where $\{G(v), \ \|v\|_2=1\}$ is the Brownian bridge indexed by halfspaces -- a centered Gaussian process with covariance function
    \begin{equation*}
    \mb EG(u)G(v) 
    = \mathrm{Cov}\l(I\{X\in H(0,u)\},I\{X\in H(0,v)\}\r),
    \end{equation*}
    where $X\sim N(0,I_d)$. A important observation is the fact that the trajectories of $W(z)$ are concave with probability $1$, as $W(z)$ is just the Legendre-Fenchel transformation of the Brownian bridge $G(v)$ -- in other words, $W(z)$ is the infimum of a family of linear functions, which is always concave.  
    We will need the following estimate on the rate of convergence of $W_n$ to $W$. 
    \begin{lemma}
    \label{lemma:rate}
        For every $C>0$, there exists $C_1>0$ such that for all $R\geq 1$, $n>C_1(R+d)^2$, one can construct on the same probability space independent random vectors $X_1,\ldots,X_n$ with standard normal distribution and a Brownian bridge $\{G(v), \ \|v\|_2=1\}$ such that 
        \[
        \sup_{\|z\|_2\leq R }|W_{n}(z) - W(z)| \leq \delta_{n} +  K \l( \frac{R^3}{n} + \l(\frac{R^2}{n}\r)^{1/4}\sqrt{d\log(n)} \r)
        \]
        with probability at least $1-2/n$, where $K=K(d)>0$, 
        \[
        \delta_{n} = C(d) n^{-\frac{1}{2(d-1)}}\log^{1+\rho(d)}(n)
        \text{ and } \rho(d)=\begin{cases}
        1, & d=2, \\ \frac{1}{2}, & d\geq 3 \end{cases}.
        \]
    \end{lemma}
    \noindent The proof of this lemma is technical, and we include it in the appendix. In what follows, we will set $R = d+t$ and 
    \begin{equation}
        \label{eq:gamma_n}
    \gamma_{n}(t):=\delta_{n} +  K \l( \frac{\l(d+t\r)^3}{n} + \l(\frac{d^2+t^2}{n}\r)^{1/4}\sqrt{d\log(n)} \r).
    \end{equation}
 Moreover, let 
 \[
 \wh{z}=\arg\max_{z\in\mathbb{R}^d}W(z),\footnote{$\wh z$ is unique with probability 1, see Lemma \ref{lemma:auxiliary} (i).} \ B(\wh{z},\alpha):=\l\{z:\; \|z-\wh{z}\|_2\leq \alpha\r\}, \ \alpha>0,
 \]
 and define 
 \[
 \mathcal{U}(\alpha):=\mathcal{U}(\wh z,\alpha) = \l\{v\in\mathbb{S}^{d-1}:\; G(v)-\frac{1}{\sqrt{2\pi}} \langle v, z\rangle\leq W(\wh z)\text{ for some } z\in B(\wh{z},\alpha)\r\}.
 \]
The following lemma essentially states the following simple fact: if a concave function $g$ defined on a convex subset of $\mb R^d$ achieves its maximum $g_U$ at a unique point $t_\ast$, then diameter of the level sets $\{t\in \mb R^d: \ g(t)\geq g_U-\tau\}$ converges to $0$ as $\tau\to 0$. 
\begin{lemma}
\label{lemma:diameter}
Let $R(\beta) =  \{ \|z\|_2\in \mb R^d: \ W(z)\geq \sup_{z\in \mb R^d} W(z) - \beta \}$. Then there exists an event $\m E$ of probability at least $1-e^{-t}$ such that for any sequence $\beta_j = o(1)$ and any $\tau>0$, 
    \begin{equation}
    \label{diameter:claim}
    \lim_{j\to\infty} \mb P\l( \l\{\mathrm{diam}(R(\beta_j))\geq \tau\r\}\cap \m E \r) = 0. 
    \end{equation}
\end{lemma}
\begin{proof}
Lemma \ref{lemma:auxiliary} (i) guarantees that $\|\wh z\|_2\leq C\sqrt{d+t}$ on event $\m E$ of probability at least $1-e^{-t}$. Assume that \eqref{diameter:claim} is false, meaning that for some increasing sequence $\{j_i\}_{i\geq 1}$, 
\[
\lim_{i\geq 1}\mb P\l( \l\{\mathrm{diam}(R(\beta_{j_i}))\geq \tau \r\}\cap \m E\r) \geq c>0.
\]
Since the sets $R(\beta_{j_i})$ are nested, it implies that 
\[
\mb P\l( \l\{\mathrm{diam}\l(\bigcap_{i\geq 1} R\l(\beta_{j_i}\r)\r)\geq \tau \r\}\cap \m E\r)>0.
\]
But any element of the set $\bigcap_{i\geq 1} R(\beta_{j_i})$ must belong to $\arg\max_{z} W(z)$, contradicting the a.s. uniqueness of the latter. 
The claim follows. 
\end{proof}
Next, we state another fact regarding the upper level sets of the process $W(z)$. 
\begin{lemma}
 \label{lemma:diameter-new}
    Suppose that there exists $\tau>0$ such that for any $z\in\mathbb{R}^d$, we can find $u\in\mathcal{U}(\alpha)$ and $z_0\in B(\wh z,\alpha)$ satisfying
    \[
    \frac{u^{\top}(z-z_0)}{\|z-z_0\|_2}> \tau \text{ and } G(u)-\frac{1}{\sqrt{2\pi}}\langle u,z_0\rangle\leq W(\wh{z}).
    \]
    Then for any $\phi>0$, 
    \begin{align*}
        \l\{z: \,W(z)\geq \sup_{z\in \mb R^d} W(z) - \phi\r\}\subseteq\l\{z: \;\l\|z-\wh{z} \r\|_2\leq \alpha+\frac{\sqrt{2\pi}}{\tau}\phi\r\}.
    \end{align*}
\end{lemma}

\begin{proof}
    Since $\wh{z}\in B(\wh{z},\alpha)$ for any $\alpha>0$, 
    \begin{align*}
        \l\{z: \,W(z)\geq \sup_{z\in \mb R^d} W(z) - \phi\r\}\subseteq \l\{z: \,
            W(z)\geq \inf_{z\in B(\wh{z},\alpha)} W(z) - \phi\r\}.
    \end{align*}
    Take any $z$ such that $\|\wh{z}-z_0\|_2\leq \alpha$. By assumption of the lemma, there exists $v\in\mathcal{U}(\alpha)$ and $z_0\in B(\wh z, \alpha)$ satisfying 
    \[
    \frac{v^{\top}(z-z_0)}{\|z-z_0\|_2}>\tau,\quad G(v)-\frac{1}{\sqrt{2\pi}} \langle v, z_0\rangle\leq W(\wh{z}).
    \]
    Therefore,
    $G(v)-\frac{1}{\sqrt{2\pi}} \langle v, z_0\rangle\leq G(v)-\frac{1}{\sqrt{2\pi}} \langle v, z\rangle+\phi$, leading to the inequality 
    \[
    \|z-z_0\|_2\leq \sqrt{2\pi}\frac{\phi}{\tau}.
    \]
    Since $\|\wh{z}-z_0\|_2\leq \alpha$, the claim follows.
\end{proof} 
We will now deduce the claims of the theorem.  
It follows from the facts established above that there exists an event $\Theta$ of probability at least $1-\frac{2}{n}-2e^{-t}$ such that the claims of Lemma \ref{lemma:rate} holds on $\Theta$, and moreover the condition $R_{n}(2\eps_n)\subseteq B\l(0,C_1\sqrt{\frac{d+t}{n}}\r)$ is satisfied. 
Therefore, on $\Theta$
    \begin{multline}
    \label{eq:R_n}
            R_n(\eps_n) = \l\{\frac{z}{\sqrt n}: \|z\|_2\leq C_1\sqrt{d+t} \text{ and } W_{n}(z)\geq \sup_{z\in \mb R^d} W_{n}(z) - 2\sqrt{n}\eps_{n}\r\}
            \\
            \subseteq \l\{\frac{z}{\sqrt{n}}: \|z\|_2\leq C_1\sqrt{d+t}, \,
            W(z)\geq \sup_{z\in \mb R^d} W(z) - 2\sup_{\|z\|\leq C_1\sqrt{d+t}}\l|W_{n}(z)-W(z)\r| - 2\sqrt{n}\eps_{n}\r\}
            \\
            \subseteq\l\{\frac{z}{\sqrt{n}}: \|z\|_2\leq C_1\sqrt{d+t} \text{ and } W(z)\geq \sup_{z\in \mb R^d} W(z) - 2\gamma_n(t) - 2\sqrt{n}\eps_{n}\r\}:=\m A_n.
    \end{multline}
 Assume that $\eps_n=o(n^{-1/2})$, whence $\gamma_n(t) +\sqrt{n}\eps_n=o(1)$ as $n\to\infty$, implying that
    \[
    \sqrt{n}\,\mathrm{diam}(\m A_n)\cdot I_{\Theta} \to 0 \text{ in probability as } n\to\infty. 
    \]
  Since $t$ can be arbitrarily large, the claim \textbf{(ii)} of the theorem follows from Lemma \ref{lemma:diameter}.
    
  To establish claim \textbf{(i)}, we will recall part (iii) of Lemma~\ref{lemma:auxiliary} which states that with probability $1$, for any $z\in\mathbb{R}^d$ we can find $u\in\mathcal{U}(\alpha)$ and $z_0\in B(\wh z,\alpha)$ such that $\frac{u^{\top}(z-z_0)}{\|z-z_0\|_2}> 0$ and 
  \[
  G(u)-\frac{1}{\sqrt{2\pi}}\langle u,z_0\rangle\leq W(\wh{z}).
  \]
  It implies that for any $\alpha>0$ and $\beta>0$, there exists $c(\alpha,\beta)>0$ such that $\frac{u^{\top}(z-z_0)}{\|z-z_0\|_2}> c(\alpha,\beta)$ with probability at least $1-\beta$. Therefore, in view of Lemma~\ref{lemma:diameter-new}
\begin{align*}
R_n(\eps_n) \subseteq \l\{ z: \ \l\|z- \frac{\wh z}{\sqrt n} \r\|_2 \leq \frac{4}{c(\alpha,\beta)\sqrt{n}}(\gamma_n(t)\vee \sqrt{n}\eps_n) +\frac{\alpha}{\sqrt{n}}\r\}
\end{align*} 
with probability at least $1-\frac{2}{n}-2e^{-t}-\beta$.
Taking $\alpha=1/2$, we conclude that with the same probability,
\[
\operatorname{diam}(R_n(\eps_n)) \leq \frac{1}{\sqrt{n}} + \frac{4}{c(\alpha,\beta)}\l(\eps_n\vee\frac{\gamma_n(t)}{\sqrt n}\r).
\]
To deduce claim \textbf{(i)} of the theorem, it remains to observe that for $d\geq 3$, the dominating term in the expression \eqref{eq:gamma_n} for $\gamma_n(t)$ is $\delta_n$. 

Finally, to prove part \textbf{(iii)}, we use similar reasoning as before combined with the claim \textbf{(ii)} of Lemma \ref{lemma:auxiliary}. The latter states that with probability $1$, $\sup\limits_{w\ne 0}\min\limits_{u\in \m U} u^{\top}\frac{w}{\|w\|_2} < 0$ where $\m U$ is the set of extreme points of $\partial W(\wh z)$. In essence, it means that $0$ is an interior point of the subdifferential $\partial W(\wh z)$, implying that $W(z)$ decays linearly away from its maximum achieved at $\wh z$. To see this, note that in view of Lemma \ref{lemma:auxiliary} \textbf{(ii)}, for any $\beta>0$ there exists $c(\beta)>0$ such that 
\[
\min\limits_{u\in \m U} u^{\top}\frac{w}{\|w\|_2} \leq -c(\beta)
\]
with probability at least $1-\beta$. The trajectories of $W(z)$ are concave with probability $1$ as the pointwise infimum of linear functions, therefore, subgradient inequality yields that for any $z\ne \wh z$ and any $u\in \m U$,
\[
W(z) \leq W(\wh z) + u^{\top}(z - \wh z). 
\]
Choosing $u=u(z)\in \m U$ for which $u^{\top}(z - \wh z)$ is minimized, we deduce that 
\[
W(z) \leq W(\wh z) - c(\beta)\|z - \wh z\|_2
\]
with probability at least $1-\beta$. In view of \eqref{eq:R_n}, this yields that 
\begin{align*}
R_n(\eps_n) &\subseteq \l\{ \frac{z}{\sqrt{n}}:\  c(\beta)\|z - \wh z\|_2\leq 2\gamma_n(t) + 2\sqrt{n}\eps_{n} \r\}
\\
&\subseteq \l\{ z: \ \l\|z- \frac{\wh z}{\sqrt n} \r\|_2 \leq \frac{4}{c(\beta)\sqrt{n}}\l(\gamma_n(t) \vee \sqrt{n}\eps_n\r) \r\}
\end{align*}
on event of probability at least $1-2e^{-t}-\frac{2}{n}-\beta$. Setting $d=2$ in the definition \eqref{eq:gamma_n} of $\gamma_n(t)$ and observing that the term $C\frac{\log^{1/2}(n)}{n^{1/4}}$ dominates in this case, we obtain the desired conclusion. 


\section{Lower bounds for the diameter of the empirical depth regions}
\label{section:lower}

The goal of this section is to establish lower bounds on the diameter of the central regions $R_n(\eps)$, and to show that the upper bounds obtained before for the case $d=2$ are optimal in some regimes. To this end, we prove a bound on the modulus of continuity of the empirical depth process: in essence, it states that $D_n(z)$ behaves like a Lipschitz-continuous function at sufficiently large ``resolution levels.'' 
\begin{theorem}
\label{th:depth:difference}
With probability at least $1-e^{-t}$, 
\begin{multline*}
\sup_{\|z_1-z_2\|_2\leq \eps}\l| D_n(z_1) - D_n(z_2) \r| \leq \frac{\eps}{\sqrt{2\pi}}
\\
+C_1\l( \sqrt{\eps}\l(\sqrt{\frac tn}+\sqrt{\frac dn}\log^{1/2}\l( \frac{C_2}{\eps}\r)\r) + \frac{1}{n}\l( t+d\log\l( \frac{C_2}{\eps}\r)\r)\r).
\end{multline*}
\end{theorem}
The following corollary states that the central regions of empirical Tukey's depth are not too small: in particular, diameter of $R_n(\eps)$ grows at least linearly with respect to $\eps$ when $\eps$ is above a certain small threshold. 
Given two subsets $A,B$ of $R^d$, $A+B = \l\{ a+b, \ a\in A, \ b\in B\r\}$ stands for their Minkowski sum. 
\begin{corollary}
\label{corollary:diameter}
There exist absolute constants $c,c_1,C_3$ with the following property. Assume that $\frac{d+t}{n}\leq c$ for $c$ small enough. Then for any $\eps \geq C_3\frac{d+t}{n}\log\l(\frac{n}{d+t}\r)$, 
\[
R_n(\eps)\supseteq R_n(0) + B(0, c_1\eps)
\]
with probability at least $1-e^{-t}$, where $R_n(0) = \arg\max_z D_n(z)$ is the Tukey's set. 
\end{corollary}
\noindent In particular, we immediately see for $d=2$,
\[
\mathrm{diam}(R_n(\eps))\geq 2c_1\eps
\] 
holds for $\eps$ larger than $\frac{C}{n}\log(n)$. On the other hand, result of Theorem \ref{th:diameter} part \textbf{(iii)} states that with high probability,
\[
\mathrm{diam}(R_n(\eps))\leq C'\eps
\]
whenever $\eps \geq C'' n^{-3/4}\log^{1/2}(n)$. Therefore, for $\eps$ that exceeds $C'' n^{-3/4}\log^{1/2}(n)$, the bound of Theorem \ref{th:diameter} \textbf{(iii)} is sharp. 
\begin{proof}[Proof of Corollary \ref{corollary:diameter}]
It is easy to see that whenever $\eps \geq C_3\frac{d+t}{n}\log\l(\frac{n}{d+t}\r)$, the term $\frac{\eps}{\sqrt{2\pi}}$ dominates in the bound of Theorem \ref{th:depth:difference}. Therefore, $D_n(z) \geq D_n(\wh\mu_n) - \eps$ whenever $\|z-\wh\mu_n\|_2\leq c_1\eps$ for a sufficiently small constant $c_1>0$.
\end{proof}

\begin{proof}[Proof of Theorem \ref{th:depth:difference}]
Given $z_1,z_2\in \mb R^d$ such that $\|z_1-z_2\|_2\leq \eps$, let $v_1,v_2$ be unit vectors such that $D_n(z_j) = P_n^X H(z_j,v_j)$, $j=1,2$. 
Assuming without loss of generality that $D_n(z_1)\geq D_n(z_2)$, note that 
\begin{multline}
\label{eq:depth:difference}
D_n(z_1) - D_n(z_2) = \underbrace{P_n^X \l(H(z_1,v_1) - H(z_1,v_2)\r)}_{\leq 0} + P_n\l( H(z_1,v_2) - H(z_2,v_2)  \r) 
\\
\leq \sup_{\|v\|_2=1,\|z_1-z_2\|\leq \eps} \l| (P_n-P)(H(z_1,v)-H(z_2,v))\r| 
\\
+ \sup_{\|v\|_2=1,\|z_1-z_2\|\leq \eps} \l| P(H(z_1,v)-H(z_2,v))\r|.
\end{multline}
We will estimate both terms on the right side of the display above. 
Denote 
\[
F_n(z,v) = \sqrt{n}(P_n^X - P)H\l(z,v\r).
\] 
The following lemma controls the modulus of continuity of $F_n(z,v)$. 
\begin{lemma} 
\label{lemma:halfspace difference}
There exist absolute constants $C_1,C_2>0$ such that for any $\eps>0$,
    \begin{multline*}
        \sup\limits_{\|z_1-z_2\|_2\leq \eps, \|v\|_2=1} \l|F_n(z_1,v) - F_n(z_2,v)\r|
        \\
        \leq C_1\l( \sqrt{\eps}\l(\sqrt{t}+\sqrt{d}\log^{1/2}\l( \frac{C_2}{\eps}\r)\r) + \frac{1}{\sqrt{n}}\l( t+d\log\l( \frac{C_2}{\eps}\r)\r)\r)
    \end{multline*}
with probability at least $1-e^{-t}$.
\end{lemma}
\begin{proof}
For brevity, set 
\[
S:= \sup\limits_{\|z_1-z_2\|_2\leq \eps, \|v\|_2=1}\l|F_n(z_1,v) - F_n(z_2,v)\r|,
\]
where the supremum is taken over all $z_1,z_2$ such that $\|z_1-z_2\|_2\leq \eps$ and all unit vectors $v$. 
Note that 
\begin{multline}
\label{eq:variance}
\sup\limits_{\|z_1-z_2\|_2\leq \eps, \|v\|_2=1} \mb E (H(z_1,v) - H(z_2,v))^2 
\\
=\sup\limits_{\|z_1-z_2\|_2\leq \eps}\sup_{\|v\|_2=1}\l| \int_{\dotp{z_2}{v}}^{\dotp{z_1}{v}} \phi(z) dz\r| 
\leq \sup\limits_{\|z_1-z_2\|_2\leq \eps}\frac{\|z_1-z_2\|_2}{\sqrt{2\pi}} = \frac{\eps}{\sqrt{2\pi}}
\end{multline}
where $\phi(z)$ is the probability density function of the standard normal law. Bousquet's form of Talagrand's concentration inequality (Theorem 7.3 in \cite{bousquet2003concentration}) yields that 
\[
S\leq C\l(\mb ES + \sqrt{\eps}\sqrt{t} + \frac{t}{\sqrt{n}} \r)  
\]
with probability at least $1-e^{-t}$. 
It remains to estimate $\mb ES$. Since the class of all halfspaces in $\mb R^d$ has Vapnik-Chervonenkis dimension $d+1$ \citep{wellner2023weak}, a well known argument based on Dudley's entropy integral (for instance, see Theorem~3.12 in \cite{koltchinskii2011oracle}) implies that 
\[
\mb ES \leq C_1\l( \sqrt{d\eps}\log^{1/2}\l( \frac{C_2}{\eps}\r) + \frac{d}{\sqrt{n}}\log\l( \frac{C_2}{\eps}\r)\r)
\]
for some absolute constants $C_1,C_2>0$. 
\end{proof}
\noindent Next, we deduce in a way similar to \eqref{eq:variance} that 
\[
\sup_{\|v\|_2=1,\|z_1-z_2\|\leq \eps} \l| P(H(z_1,v)-H(z_2,v))\r| \leq \frac{\eps}{\sqrt{2\pi}}.
\]
Therefore, \eqref{eq:depth:difference} implies that with probability at least $1-e^{-t}$, 
\begin{multline*}
\sup_{\|z_1-z_2\|_2\leq \eps}\l| D_n(z_1) - D_n(z_2) \r| 
\\
\leq \frac{\eps}{\sqrt{2\pi}}
+C_1\l( \sqrt{\eps}\l(\sqrt{\frac tn}+\sqrt{\frac dn}\log^{1/2}\l( \frac{C_2}{\eps}\r)\r) + \frac{1}{n}\l( t+d\log\l( \frac{C_2}{\eps}\r)\r)\r)
\end{multline*}
which is the desired bound.
\end{proof}

\section{Discussion}
\label{section:discussion}

We proved that Tukey's median, and more generally affine equivariant estimators of location, are sensitive to geometry of the underlying distribution, in a sense that their performance is controlled by the effective rank of the ``shape'' matrix rather than the ambient dimension. In the adversarial contamination framework, our claim remains valid as long as the contamination proportion $\eps$ is small, namely, at most $c(d) n^{-1/2}$. Therefore, there is still a gap between the bounds for Tukey's median proved in this paper and best known guarantees for alternative estimators. This fact motivates the following natural question: \emph{what are the best guarantees for Tukey's median in the adversarial contamination framework with $\eps_n$ of order $\sqrt{\frac{r(\Sigma)}{n}}$? }

We also showed that when the data are Gaussian, diameter of the empirical Tukey's set is of order $o_P(n^{-1/2})$, and that in the special case $d=2$, the diameter is $O_P\l(n^{-3/4}\log^{1/2}(n) \r)$. Formally following the steps of our analysis in the case $d=1$, it is possible to obtain the rate for the diameter of the set of medians of order $\frac{\log^{1/2}(n)}{n}$ which is sharp up to log-factors. This fact leads us to conjecture that in the case $d=2$, our result is also nearly sharp, and that for $d\geq 3$, the diameter is of order $n^{-\frac12 -\frac{1}{2(d-1)}}$ (again up to logarithmic factors). 

One of the main technical challenges on the path to answering the stated questions about Tukey's median and diameter of Tukey's set is related to the properties of the random concave function $W(z)$ defined in \eqref{def:W(z)} as the Legendre-Fenchel transform of a Brownian bridge (see the discussion preceding the proof of Theorem \ref{th:diameter}). Specifically, let $\wh z$ be the $\arg\max$ of $W(z)$. How fast does a typical realization of $W(\wh z - z)$ decrease as a function of $\|z - \wh z\|_2$ and dimension $d$? Letting $\partial W(\wh z)$ be the subdifferential of $W(z)$ at the point $\wh z$, we can invoke the subgradient inequality to deduce that for any $u\in \partial W(\wh z)$,
\[
W(z) \leq W(\wh z) + u^T(z - \wh z).
\]
Therefore, the question about the speed of decay of $W(z)$ can be reformulated as follows: what is a high-probability lower bound for a constant $c$ such that
\[
\sup_{w\ne 0}\min\limits_{u\in \m U} u^{\top}\frac{w}{\|w\|_2}\leq -c
\]
where $\m U$ is the set of extreme points of $\partial W(\wh z)$? What if the Brownian bridge is replaced by the empirical process indexed by halfspaces? 
Let us remark that if instead of the Brownian bridge $G(v)$ in the definition of $W_G(z)$ we consider the isonormal Gaussian process $Z(v) = v^T Z$ where $Z$ has standard Gaussian distribution in $\mb R^d$, then its Legendre-Fenchel transform is simply $W_Z(z) = -\|Z-z\|_2$ and $\wh z = Z$.

Some questions related to the properties of random convex/concave functions at their extrema have been studied in the univariate case (for instance, see the papers by \citet{balabdaoui2011grenander,groeneboom1983concave}), but no multivariate extensions are readily available to the best of our knowledge.

\medskip
\noindent\textbf{Acknowledgments.}
Authors thank V.-E. Brunel and G. Lecu\'{e} for their helpful suggestions that lead to improvements of parts of the manuscript.


\bibliographystyle{apalike}
\bibliography{ref}

\appendix

\section{Proof of Theorem~\ref{th:affine}.}
\label{section:proof:affine}

\noindent Define the random vector 
\[
W:=\Sigma^{-1/2}(\widetilde\mu_n(X_1,\ldots,X_n) - \mu),
\]
and observe that we need to estimate $\|\Sigma^{1/2}W\|_2$. Moreover, set $Z_j = \Sigma^{-1/2}(X_j - \mu), \ j=1,\ldots,n$. Let us record the following identity:  
\[
\Sigma^{1/2}W = \|W\|_2\Sigma^{1/2}\l[\frac{W}{\|W\|_2}\r],
\]
where $U:=\frac{W}{\|W\|_2}$ is defined as $0$ if $W=0$. We claim that conditionally on $W\ne 0$, $U$ has uniform distribution over the sphere of radius $1$. Indeed, for any orthogonal matrix $Q$, the vector $QZ_1$ has the same distribution as $Z_1$ in view of elliptical symmetry, which implies that 
$\widetilde\mu_n(QZ_1,\ldots,QZ_n)$ and $\widetilde\mu_n(Z_1,\ldots,Z_n)$ have the same distribution. But 
\begin{align*}
&\widetilde\mu_n(QZ_1,\ldots,QZ_n) = Q\widetilde\mu_n(Z_1,\ldots,Z_n), 
\\
&\widetilde\mu_n(Z_1,\ldots,Z_n) = \Sigma^{-1/2}(\widetilde\mu_n(X_1,\ldots,X_n) - \mu) = W,
\end{align*}
due to affine equivariance, hence $W \stackrel{d}{=} QW$.
It implies that the conditional distribution of $\frac{W}{\|W\|}$ given that $W\ne 0$ is indeed uniform.

The function $w\mapsto \|\Sigma^{1/2}w\|_2$ Lipschitz continuous with Lipschitz constant equal to $\|\Sigma\|^{1/2}$, therefore we 
can apply L\'{e}vy's lemma \citep{levy1951problemes}, again conditionally on $W\ne 0$, to get that 
\[
\|\Sigma^{1/2}U\|_2 \leq \mb E \|\Sigma^{1/2}U\|_2 + \sqrt{\frac{2\|\Sigma\|^{1/2}t}{d}}
\]
with probability at least $1-e^{-t}$ (this version of the inequality with sharp constants is due to \cite{aubrun2024optimal}). The same result also holds unconditionally since $U=0$ if $W=0$. It remains to observe that 
\[
\mb E\|\Sigma^{1/2}U\|_2 \leq \mb E^{1/2}\|\Sigma^{1/2}U\|^2_2 = \sqrt{\frac{\tr(\Sigma)}{d}},
\]
hence we conclude by the union bound that
\begin{equation*}
\|\Sigma^{1/2}W\|_2 = \|W\|_2 \|\Sigma^{1/2}U\|_2 
\leq e(n,t,d)\l( \sqrt{\frac{\tr(\Sigma)}{d}} + \sqrt{\frac{2\|\Sigma\|^{1/2}t}{d}}\r)
\end{equation*}
with probability at least $1-p(t) - e^{-t}$.

\begin{remark}
Spherical symmetry plays a crucial role in the proof of the theorem. Due to the affine invariance of the depth function, it is not possible to capture the dependence of the error on the effective rank by simply comparing the empirical and the true depth. For example, let $X_1,\ldots,X_n$ be i.i.d. $N(0,\Sigma)$ random vectors where $\Sigma$ is non-singular, and note that for any halfspace $H(0,u)$, $P(X_1\in \HS{0}{u})=\frac{1}{2}$. Therefore,  
\begin{equation*}
    \sup_{\|u\|_2=1}\left| \frac{1}{n}\sum_{j=1}^n I\{X_j \in \HS{0}{u}\} - \frac{1}{2} \right| 
    = \sup_{\|u\|_2=1}\left|  \frac{1}{n}\sum_{j=1}^n I\{\Sigma^{-1/2}X_j \in \HS{0}{u}\} - \frac{1}{2} \right|,
\end{equation*}
hence the distribution of the supremum does not depend on $\Sigma$.
\end{remark}

\section{Proof of Lemma \ref{lemma:rate}.}
\label{section:proof:rate}

The key step in the argument involves construction of a coupling between the empirical process 
\[
F_n(v)=\sqrt{n}(P_n^X - P)\HS{0}{v}
\] 
and the Brownian bridge $G(v)$ indexed by the indicator functions of halfspaces passing through the origin. We state this result next. 
\begin{lemma}
    \label{lemma:strong:approximation}
Let $d\geq 2$. For any $n\geq 1$, there exists a sequence $X_1,\ldots,X_{n}$ of independent standard normal random vectors and the Brownian bridge $G(v)$ defined on the same probability space such that
    \begin{equation}
    \label{eq:strongapprox}
    \sup_{\|v\|_2=1} \l| F_n(v) - G(v)\r| \leq \delta_{n}(t):=C(d) n^{-\frac{1}{2(d-1)}}\log^{\rho(d)}(n)\l( s+\log(n)\r)
    \end{equation}
    with probability at least $1-e^{-s}$, where $C(d)$ depends only on $d$ and $\rho(d)=\begin{cases}
        1, & d=2, \\ \frac{1}{2}, & d\geq 3
    \end{cases}$.
\end{lemma}
\begin{proof}
The claim follows from the strong approximation result by \citet[][Theorem 3.4]{koltchinskii1994komlos} once we verify that the required assumptions hold. First, recall that the Vapnik-Chervonenkis dimension of halfspaces passing through the origin equals $d$ \citep[see][]{wellner2023weak}, hence condition (3.9) in \cite{koltchinskii1994komlos} holds with $\rho = 1$; see \cite{dudle1987universal} for the details. 

Next, observe that
\begin{multline*}
\sup_{\|v\|_2=1}\l| \sqrt{n}(P_n^X-P)H\l(0,v\r)-G(v)\r|\\=\sup_{\|v\|_2=1}\l| \frac{1}{\sqrt n}\sum_{i=1}^n\l\{I\l\{\langle X_i,v\rangle\geq 0\r\}-\frac{1}{2}\r\}-G(v)\r|\\
\\=\sup_{\|v\|_2=1}\l| \frac{1}{\sqrt n}\sum_{i=1}^n\l\{I\l\{\l\langle \frac{X_i}{\|X_i\|_\infty},v\r\rangle\geq 0\r\}-\frac{1}{2}\r\}-G(v)\r|.
\end{multline*}
If $X\sim N(0,I_d)$, then with probability one, the random vector $Z:=\frac{X}{\|X\|_\infty}$ belongs to the boundary $\partial[-1,1]^d = \l\{ x\in \mb R^d: \ \|x\|_\infty = 1 \r\}$ of the unit cube with faces denoted by $F_{0},\ldots,F_{2d-1}$ (note that each $F_j$ is itself isometric to $[-1,1]^{d-1}$). Following \cite{koltchinskii1994komlos}, we will construct the binary expansion $\l\{\Delta_j\r\}_{j\geq 0}$ of $\partial[-1,1]^d$ that satisfies the following properties: 
\begin{enumerate}
    \item $\Delta_j = \l\{ \Delta_{j,k}: \ k=0,\ldots, 2d\cdot 2^j -1  \r\}$;
    \item $\Delta_{0} = \l\{ F_0,\ldots,F_{2d-1}\r\}$, $\Delta_{j,k}=\Delta_{j+1,2k} \cup \Delta_{j+1,2k+1}$, $\Delta_{j+1,2k} \cap \Delta_{j+1,2k+1} = \emptyset$ for $k=0,\ldots, k=0,\ldots, 2d\cdot 2^j -1$;
    \item $\mb{P}(Z\in\Delta_{j+1,2k})=\Delta_{j+1,2k+1}= \frac{2^{-(j+1)}}{2d}$. 
\end{enumerate}
Each $\Delta_{j,k}$ is a parallelepiped with sides parallel to the coordinate axes. Given $\Delta_{j,k}$, we construct $\Delta_{j+1,2k}$ and $\Delta_{j+1,2k+1}$ by choosing the longest side and splitting $\Delta_{j,k}$ along it so that the condition (3) is satisfied. Formally, suppose that $\Delta_{j,k}\subseteq F_0 = \{-1\}\times [-1,1]^{d-1}$ is such that 
$\Delta_{j,k} = \{-1\}\times [a_2,b_2]\times\cdots \times [a_d,b_d]$. Without loss of generality, assume that $|a_d,b_d|=\max_{i=2,\ldots,d}|b_i-a_i|$. Then 
\[
\Delta_{j+1,2k} = \{-1\}\times [a_2,b_2]\times\cdots \times [a_d,z_d], \quad \Delta_{j+1,2k+1} = \{-1\}\times [a_2,b_2]\times\cdots \times [z_d,b_d]
\]
where $z_d$ is such that $\mb{P}(Z\in\Delta_{j+1,2k})=\Delta_{j+1,2k+1}= \frac{2^{-(j+1)}}{2d}$. We claim that 
\[
\operatorname{diam}(\Delta_{j,k})\leq C(d) 2^{-j/(d-1)}.
\]
To this end, let $h_Z(\cdot)$ be the probability density function of $Z$ with respect to the standard volume measure on $\partial[-1,1]^{d}$. 
\begin{lemma}
\label{lemma:density:bound}
The following inequalities hold:
\begin{equation}
    \label{eq:density:bound:2}
\frac{1}{2^{d+1}\sqrt{\pi} d^{d/2}}=:L_h \leq \inf_{z\in \partial[-1,1]^d} h_Z(z) \leq \sup_{z\in \partial[-1,1]^d} h_Z(z) \leq U_h:=\frac{\Gamma(d/2)2^{d-1}}{2\pi^{d/2}}.
\end{equation}
\end{lemma}
\noindent The proof is given in Section \ref{proof:density:bound} below. Note that inequality \eqref{eq:density:bound:2} yields that 
\[
\frac{\sup_{z\in \partial[-1,1]^d} h_Z(z)}{\inf_{z\in \partial[-1,1]^d} h_Z(z)}\leq (Cd)^d
\]
for an absolute constant $C>0$. Moreover, given $\Delta_{j,k} = \{-1\}\times [a_2,b_2]\times\cdots \times [a_d,b_d]$, observe that 
\[
\frac{1}{2U_h} \leq \frac{z_d - a_d}{b_d - a_d} \leq \frac{1}{2L_h}
\]
where $a_i,b_i$ are defined same as before, whence $\frac{\max_{i=2,\ldots,d}|b_i-a_i|}{\min_{i=2,\ldots,d}|b_i-a_i|}\leq \frac{U_h}{L_h}\leq (Cd)^d$. Therefore,  
\begin{multline*}
\mb{P}\l(Z\in \Delta_{j,k}\r) = \frac{2^{-j}}{2d}\int\limits_{\prod_{i=2}^d [a_i,b_i]} h_Z(z) dz \geq L_h \prod_{i=2}^d |b_i-a_i| \geq L_h \l( \min_i |b_i - a_i|\r)^{d-1}
\\
\geq L_h\l( \frac{\max_i |b_i - a_i|}{(Cd)^d}\r)^{d-1},
\end{multline*}
and we deduce that 
\begin{align*}
\operatorname{diam}(\Delta_{j,k})&\leq (d-1)^{1/2}\max_{i=2,\ldots,d} |b_i - a_i| \leq 
(d-1)^{1/2} (Cd)^d(2d L_h)^{-\frac{1}{d-1}}  2^{-j/(d-1)}
\\
&=: C(d)\, 2^{-j/(d-1)}.
\end{align*}
Next, the binary expansion $\{\Delta_j, \ j\geq 0\}$ is associated with the Haar basis 
\begin{align*}
h_{0,k}(x) &= I\{x\in F_k\}, \ k=0,\ldots,2d-1, 
\\
h_{j,k}(x) & = 2^{j/2}\l( I\{x\in \Delta_{j+1,2k} \} - I\{x\in \Delta_{j+1,2k+1} \} \r), \ j\geq 1, \ k=0,\ldots,2d\cdot 2^j - 1.
\end{align*}
Given a function $g$ such that $\mb Eg^2(Z)<\infty$ and an integer $j\geq 1$, we define 
\[
g_j = \sum_{i=0}^{j-1}\sum_{k=0}^{2d\cdot 2^i-1}\langle g,h_{j,k}\rangle h_{j,k}
\]
where $\langle g_1,g_2\rangle = \mb E \l[g_1(Z) g_2(Z)\r]$. In order to apply Theorem 3.4 in \cite{koltchinskii1994komlos}, it remains to find an upper bound for the approximation error 
\[
\mb E \l( f_v(Z) - f_{v,j}(Z)\r)^2 \text{ where }f_v(Z):=I\l\{\l\langle Z,v\r\rangle\geq 0\r\}.
\]
To do so, we first observe that 
\begin{equation*}
    \mb E \l( f_v(Z) - f_{v,j}(Z)\r)^2 = \int\limits_{\partial[-1,1]^{d}} (f_v(z) - f_{v,j}(z))^2 h_Z(z)dz
        \leq U_h \int\limits_{\partial[-1,1]^{d}} (f_v(z) - f_{v,j}(u))^2 du.
\end{equation*}
Following \citet[][p.110]{koltchinskii1994komlos}, we deduce that 
\begin{multline*}
\int\limits_{\partial[-1,1]^{d}} (f_v(z) - f_{v,j}(u))^2 du \leq 
\sum_{k=0}^{2^{j+1}d-1}2^N d\int_{\Delta_{j,k}} \int_{\Delta_{j,k}} |f_v(z)-f_v(y)|^2 dz\; dy
\\
 \leq 2^N d \sup_{\|\mathbf{r}\|_2\leq C(d)2^{-j/(d-1)}}\int_{\partial[-1,1]^{d}} |f_v(z+\mathbf{r})-f_v(z)|^2 dz \leq C_1(d)\omega^2\l(f_v;C(d)2^{-j/(d-1)}\r)
\end{multline*}
where the modulus of continuity $\omega(f_v,h)$ is defined via
\begin{equation}
    \label{eq:modulus}
    \omega(f_v;h) = \sup_{\|\mathbf{r}\|_2\leq h}\l( \int_{\partial[-1,1]^{d-1}} (f_{v}(z + \mathbf{r}) - f_{v}(z))^2 dz\r)^{1/2}.
    \end{equation}
Without loss of generality, assume that $\langle \mathbf{r},v\rangle\geq 0$. It is easy to see that $(f_{v}(z + \mathbf{r}) - f_{v}(z))^2 = 1$ iff $\langle z,v \rangle\in [-\langle \mathbf{r},v\rangle,0]$. Assume that $d\geq 3$, and let $F$ be a face of $\partial[-1,1]^d$. Recall that K. Ball's theorem states that the largest possible $(d-2)$-dimensional) volume of the set $\mathcal A_t:=F\cap \l\{ z: \, \langle z,v \rangle = t \r\}$ is $2^{d-2}\sqrt{2}$. 
This immediately implies that  
\[
\int_{\partial[-1,1]^{d-1}} (f_{v}(z + \mathbf{r}) - f_{v}(z))^2 dz \leq 2^{d-1}d\sqrt{2}\l|\langle \mathbf{r},v\rangle\r|,
\]
hence $\omega^2(f_v;h) \leq 2^{d-1}d\sqrt{2}h$. In the case $d=2$, it is easy to see that the same conclusion holds since the length of a line segment $F\cap \l\{ z: \, \langle z,v \rangle \in [-\mathbf{r},0] \r\}$ is at most $\mathbf{r}$. 
Finally, Theorem 3.4 in \cite{koltchinskii1994komlos} applies with $\rho =1$ and $\beta = \frac{1}{d-1}$, and yields the desired inequality.
\end{proof}


We are now ready to proceed with the proof of the main claim of  Lemma \ref{lemma:rate}. Recall the definitions \eqref{def:W_n(z)} and \eqref{def:W(z)} of $W_n(z)$ and $W(z)$ respectively, and consider two possibilities: (a) $W_{n}(z) - W(z)\geq 0$ and (b) $W_{n}(z) - W(z) < 0$. Observe that the infimum in the definitions of $W_{n}(z)$ and $W(z)$ is attained, almost surely, at some unit vectors $v_{n}(z)$ and $v(z)$. We will only consider case (a) below since the reasoning in case (b) is very similar. The following relation is evident:
  \begin{multline}
    \label{eq:decomp}
        0\leq W_{n}(z) - W(z) = \sqrt{n}\l(\EHS{n}{\frac{z}{\sqrt{n}}}{v_n(z)}  - \frac12\r) - \l( G(v(z),0) - \frac{1}{\sqrt{2\pi}} \langle z, v(z)\rangle \r)
        \\
        = \underbrace{\sqrt{n}\l(\EHS{n}{\frac{z}{\sqrt{n}}}{v_{n}(z)} - \frac12 \r) - \sqrt{n}\l(\EHS{n}{\frac{z}{\sqrt{n}}}{v(z)}  - \frac12\r)}_{\leq 0} 
        \\
       +\underbrace{\sqrt{n}(P_n-P)\l(H\l(\frac{z}{\sqrt{n}},v(z)\r) - H\l(0,v(z)\r)\r)}_{=\mathrm{I}} + \underbrace{\sqrt{n}\l(\PHS{\frac{z}{\sqrt{n}}}{v(z)} - \frac12 \r) +  \frac{1}{\sqrt{2\pi}} \langle z, v(z)\rangle}_{=\mathrm{II}}
       \\
       +\underbrace{\sqrt{n}(P_{n}^X - P)\HS{0}{v(z)} - G\l(v(z),0\r)}_{=\mathrm{III}}
    \end{multline}
 We will now estimate the terms $\mathrm{I},\,\mathrm{II}$ and $\mathrm{III}$. First, we apply Lemma \ref{lemma:halfspace difference} to deduce that with probability at least $1-e^{-t}$,
 \begin{multline*}
     \sup_{\|z\|_2\leq R} \l|\sqrt{n}(P_n-P)\l(H\l(\frac{z}{\sqrt{n}},v(z)\r) - H\l(0,v(z)\r)\r)\r|  
     \\
     \leq 
     C\l( \frac{\sqrt{R}}{n^{1/4}}\l(\sqrt{t}+\sqrt{d}\log^{1/2}\l( \frac{C_1n}{R^2}\r)\r) + \frac{1}{\sqrt{n}}\l( t+d\log\l( \frac{C_1 n}{R^2}\r)\r)\r).
 \end{multline*}
 When $t=\log(2n)$, the last expression does not exceed 
 \(
 C\l(\frac{\sqrt{R}}{n^{1/4}}\sqrt{d\log(n)} + \frac{d\log(n)}{\sqrt{n}}\r).
 \) Note that when $n\geq Cd^2$ and $R\geq 1$, the first term in the sum dominates, and we get that with probability at least $1-\frac{1}{2n}$,
 \[
 \mathrm{I}\leq C\frac{\sqrt{R}}{n^{1/4}}\sqrt{d\log(n)}
 \]
    Recalling that all 1-dimensional projections of $X$ are standard normal, we deduce that 
    \begin{multline*}
    \mathrm{II} = \sqrt{n}\l(\PHS{\frac{z}{\sqrt{n}}}{v(z)} - \frac12 \r) +  \frac{1}{\sqrt{2\pi}} \langle z, v(z)\rangle 
    = \frac{1}{\sqrt{2\pi}} \langle z, v(z)\rangle  - \sqrt{n}\int_{0}^{\l\langle \frac{z}{\sqrt{n}}, v(z)\r\rangle}\phi(t)dt 
    \leq C\frac{\|z\|_2^3}{n}
    \end{multline*}
    where we used the fact that $\phi(t)\geq \frac{1}{\sqrt{2\pi}}(1-t^2)$ for small $t$. 
    Finally, inequality \eqref{eq:strongapprox} implies that 
    \[
    \mathrm{III}\leq\sup_{\|v\|_2=1}\l|\sqrt{n}(P_{n}^X - P)\HS{0}{v} - G\l(v\r)\r| \leq C(d) n^{-\frac{1}{2(d-1)}}\log^{1+\rho(d)}(n)
    \]
    with probability at least $1-\frac{1}{2n}$, where $\rho(d)=\begin{cases}
        1, & d=2, \\ \frac{1}{2}, & d\geq 3
    \end{cases}$.  
    In case (b), the reasoning is very similar. Collecting the estimates of each term in \eqref{eq:decomp}, we deduce the desired result.
\subsection{Proof of Lemma \ref{lemma:density:bound}.}
\label{proof:density:bound}

  Note that we can equivalently define the vector $Z = \frac{X}{\|X\|_\infty}$ as $Z = \frac{Y}{\|Y\|_\infty}$ where $Y = \frac{X}{\|X\|_2}$ so that $Y$ has uniform distribution on $S^{d-1}$. Next, given $x\in\mathbb{S}^{d-1}$ and $r>0$, set 
    \[
    U(x,r):=\l\{z\in\mathbb{S}^{d-1}:\; \|x-z\|_2\leq r\r\}.
    \]
    Let $V(x,r)\subset\partial[-1,1]^{d}$ be the image of $U(x,r)$ under the mapping $z\mapsto \frac{z}{\|z\|_\infty}$: 
    \begin{align*}
        V(x,r):=\l\{\frac{z}{\|z\|_{\infty}}: \ z\in U(x,r)\r\}.
    \end{align*}
    It is sufficient to consider $x$ such that $\frac{x}{\|x\|_\infty}$ belongs to the relative interior of a face of $\partial[0,1]^d$. The density $h_Z(x/\|x\|_\infty)$ satisfies the following: 
    \begin{equation}
    \label{eq:density:bound:1}
        h_Z\l(\frac{x}{\|x\|_\infty}\r)=\lim_{r\to 0}\frac{\mathbb{P} (Z\in V(x,r))}{\vol(V(x,r))}=
        \lim_{r\to 0}\frac{\mathbb{P} (Y\in U(x,r))}{\vol(U(x,r))}\cdot\frac{\vol(U(x,r)) }{\vol(V(x,r))},
    \end{equation}
    where $\vol(\cdot)$ denotes (with some abuse of notation) the volume measure on $\mb S^{d-1}$ or $\partial[-1,1]^{d}$ depending on the context. Clearly, 
    \[
    h_Y(x) := \lim_{r\to 0}\frac{\mathbb{P} (Y\in U(x,r))}{\vol(U(x,r))} = \frac{\Gamma(d/2)}{2\pi^{d/2}}
    \]
    is the density of $Y$ corresponding to the uniform distribution on $\mb S^{d-1}$. 
    Hence, it suffices to bound the ratio $\frac{\vol(U(x,r)) }{\vol(V(x,r))}$. For small $r$, the volume measure on the sphere can be approximated by the Lebesgue measure on $\mb R^{d-1}$, implying that
    \[
    \vol(U(x,r)) = (1 + o(1))\frac{\pi^{(d-1)/2}r^{d-1}}{\Gamma\l(\frac{d+1}{2}\r)} \text{ as } r\to 0.
    \]
   On the other hand, note that
    \begin{align*}
        \l\|\frac{z}{\|z\|_{\infty}}-\frac{x}{\|x\|_{\infty}} \r\|_{\infty}\leq \frac{\|z-x\|_2}{\|z\|_{\infty}}+\frac{\|x\|_{\infty}\|x-z\|_2}{\|x\|_{\infty}\|z\|_{\infty}}\leq 2\sqrt{d}\eps,
    \end{align*}
    where we used the fact that $\|z\|_\infty\geq d^{-1/2}$ for $z\in \mb S^{d-1}$.
    Therefore,  
    \[
    V(x,r)\subseteq \l\{z\in\partial[-1,1]^d, \l\|z-\frac{x}{\|x\|_{\infty}}\r\|_{\infty}\leq 2\sqrt{d}r\r\},
    \]
    implying that 
    \begin{equation}
        \label{eq:volume:bound1}
    \vol(V(x,r))\leq (2\sqrt{d}r)^{d-1}.
    \end{equation}
    Next, if $x',z\in \partial[-1,1]^d$ are such that $\|z-x'\|_2\leq r/2$, then 
    \[
    \l\| \frac{z}{\|z\|_2} - \frac{x'}{\|x'\|_2}\r\|_2\leq r,
    \]
    hence 
    \[
    V(x,r)\supseteq \l\{z\in\partial[-1,1]^d, \l\|z-\frac{x}{\|x\|_{\infty}}\r\|_2\leq \frac{r}{2}\r\}
    \]
    and 
    \begin{equation}
        \label{eq:volume:bound2}
    \vol(V(x,r))\geq \frac{\pi^{(d-1)/2}r^{d-1}}{2^{d-1}\Gamma\l(\frac{d+1}{2}\r)}
    \end{equation}
    for sufficiently small $r>0$. It remains to combine inequalities \eqref{eq:volume:bound1}, \eqref{eq:volume:bound2} with \eqref{eq:density:bound:1} conclude that
    \begin{align*}
         \frac{\Gamma(d/2)}{\Gamma\l(\frac{d+1}{2}\r)}\frac{1}{2^d\sqrt{\pi} d^{(d-1)/2}}
         \leq h_Z\l(\frac{x}{\|x\|_\infty}\r) 
         \leq \frac{\Gamma(d/2)2^{d-1}}{2\pi^{d/2}},
    \end{align*}
    The lower bounds can be simplified using the bound $\frac{\Gamma(d/2)}{\Gamma\l( \frac{d+1}{2}\r)}\geq (d+1)^{-1/2}\geq \frac{1}{2\sqrt{d}}$ that follows from inequalities due to \cite{gautschi1959some}, implying the final form of the bound.

\subsection{Auxiliary results.}
\label{sec:auxiliary}

Recall the definition \eqref{def:W(z)} of the process $W(z)$. Moreover, let 
\begin{align*}
\wh z &\in \arg\max_{z\in \mb R^d} W(z), 
\\
\mathcal{U}&:=\mathcal{U}(\wh z) = \arg\min_{v:\|v\|_2=1} \l\{G(v)-\frac{1}{\sqrt{2\pi}} \langle v,\wh z\rangle\r\},
\\
B(\wh{z},\alpha)&:=\l\{z:\; \|z-\wh{z}\|_2\leq \alpha\r\}, 
\\
\mathcal{U}(\alpha)&:=\mathcal{U}(\wh z,\alpha) = \l\{v\in\mathbb{S}^{d-1}:\; G(v)-\frac{1}{\sqrt{2\pi}} \langle v, z\rangle \leq W(\wh z)\text{ for some } z\in B(\wh{z},\alpha)\r\}.
\end{align*}.
\begin{lemma}
\label{lemma:auxiliary}
The following statements hold: 
\begin{enumerate}
\item[(i)] With probability $1$, $\wh z$ is unique. Moreover, there exists event $\m E$ of probability at least $1-e^{-t}$ such that on this event, $\|\wh z\|_2\leq C\sqrt{d+t}$;
\item[(ii)] With probability $1$, for any $w\in \mb R^d$, $w\ne 0$, there exists $u\in \mathcal{U}$ such that $u^{\top}w\leq 0$. Moreover, if $d=2$, then the inequality can be strengthened to $\sup\limits_{w\ne 0}\min\limits_{u\in \m U} u^{\top}\frac{w}{\|w\|_2} < 0$. 
\item[(iii)] With probability $1$, for any $\alpha>0$
\[
\min_{v\in\mathbb{S}^{d-1}}\max_{u\in\mathcal{U}(\alpha)}v^{\top}u>0.
\]
Moreover, for any $z\in\mathbb{R}^d$, with probability $1$ there exists $u\in\mathcal{U}(\alpha)$ and $z_0\in B(\wh z,\alpha)$ such that 
\[
\frac{u^{\top}(z-z_0)}{\|z-z_0\|_2}> 0 \text{ and } G(u)-\frac{1}{\sqrt{2\pi}}\langle u,z_0\rangle\leq W(\wh{z}).
\]
\end{enumerate}
\end{lemma}
\begin{proof}
\textbf{(i)} Lemma 3.8 in \cite{masse2002asymptotics} states that with probability 1, $\arg\max_{z\in \mb R^d}W(z)$ is unique. To establish the bound on $\|\wh z\|_2$, note that in view of the Gaussian concentration inequality \citep{cirel1976norms}, 
\[
\inf_{\|v\|_2=1}G(v)\geq \mb E \inf_{\|v\|_2=1}G(v) - \sqrt{2t}\sup_{\|v\|_2=1}\mathrm{var}^{1/2}(G(v))
\]
with probability at least $1-e^{-t}$. 
A standard argument\footnote{For example, it suffices to combine Theorem 2.37 with Theorem 4.47 or 4.52 in \citep{dudley2014uniform}.} based on Dudley's entropy integral bound \citep{dudley1967sizes} and the fact that the Vapnik-Chervonenkis dimension of the class of halfspaces passing through the origin equals $d$ implies that 
\[
\mb E \inf_{\|v\|_2=1}G(v)\geq -C_1\sqrt{d}.
\]
Since $\mathrm{var}(G(v))=1/4$ for all $v$, we deduce that 
\[
\inf_{\|v\|_2=1}G(v)\geq -C_2\sqrt{d+t}
\]
on event $\m E$ of probability at least $1-e^{-t}$. We conclude that on $\m E$,
\[
W(0) = \inf_{\|v\|_2=1}G(v) \geq -C_2\sqrt{d+t}.
\]
Moreover, for $\|z\|_2> C\sqrt{d+t}$, we have
\[
W(z)\leq -\frac{1}{\sqrt{2\pi}}\|z\|_2 +\sup_{\|v\|_2=1}G(v) < C_2 \sqrt{d+t},
\]
whenever $C$ is sufficiently large. Therefore, $\|\wh z\|_2 \leq C\sqrt{d+t}$ on $\m E$. 

\textbf{(ii)} The first claim is established in the course of the proof of Lemma 3.8 in \cite{masse2002asymptotics}, also see \cite{nolan1999min}. 
Assume now that $d=2$. \cite{masse2002asymptotics} and \cite{nolan1999min} show that the set $\m U$ must contain at least 3 elements. Trajectories of the process $W(z)$ are concave with probability $1$ as the pointwise infimum of linear functions, moreover, the subdifferential $\partial W(\wh z)$ is the convex hull $\mathrm{co}(\m U)$ of the set $\m U$ \citep{boyd2004convex}. The necessary condition for $\wh z$ to be the maximizer of $W(z)$ is $0\in \partial W(\wh z)$. If $0$ is on the boundary of $\mathrm{co}(\m U)$, then there exists a chord passing through $0$ contained in $\mathrm{co}(\m U)$, implying that there exists a unit vector $\hat v$ such that $\l\{ \hat v,-\hat v \r\}\subset \m U$ (see Figure \ref{figure:subdifferential2}). But in this case 
\begin{equation}
\label{eq:chord}
M(\wh z,u):=G(u) - \frac{1}{\sqrt{2\pi}} \langle \wh z, u\rangle = 0  
\end{equation}
for all unit vectors $u$, in which case $\m U = S^{2}$ and the claim readily holds. To verify \eqref{eq:chord}, observe that, since $G(-u) = -G(u)$ with probability $1$, $M(z,-u) = -M(z,u)$ for all $z,u$. Therefore,
\[
W(\wh z) = M(\wh z,\hat v) = \inf_{\|u\|_2=1} M(\wh z,u)\leq \sup_{\|u\|_2=1} M(\wh z,u) = M(\wh z,-\hat v) = W(\wh z),
\]
implying the claim. 
\begin{figure}[t]
    \centering
   \begin{subfigure}[b]{0.45\textwidth}
      \centering
       \includegraphics[width=0.7\textwidth]{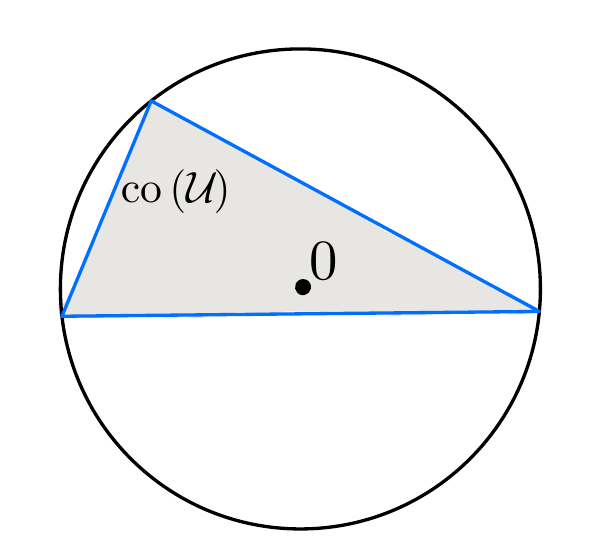}     
       \caption{Depiction of the case when $0$ is in the interior of $\mathrm{co}(\m U)$.}
         \label{figure:subdifferential1}
    \end{subfigure}
    \begin{subfigure}[b]{0.45\textwidth}
        \centering
        \includegraphics[width=0.7\textwidth]{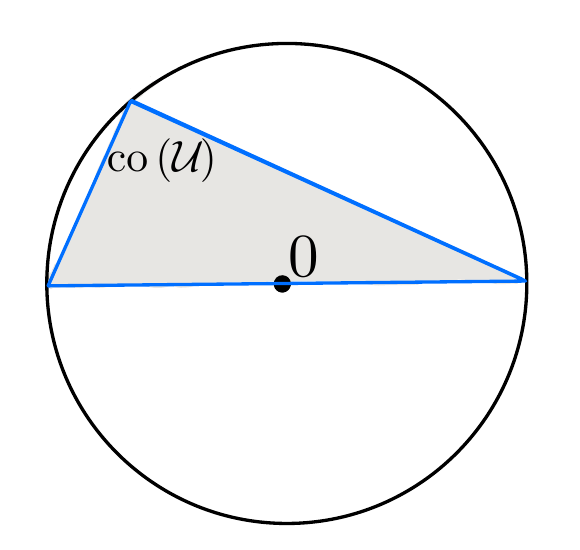}   
        \caption{Depiction of the case when $0$ is on the boundary of $\mathrm{co}(\m U)$.}
        \label{figure:subdifferential2}
   \end{subfigure}
   \caption{\mbox{Subdifferential of $W(\wh z)$.}}
  \label{figure:subdifferential}
 \end{figure}

\noindent Finally, when $0$ is in the interior of $\mathrm{co}(\m U)$ (Figure \ref{figure:subdifferential1}), then 
\[
\sup_{w\ne 0}\min\limits_{u\in \m U} u^{\top}\frac{w}{\|w\|_2} = -\mathrm{dist}\l(0,\mathrm{bd}(\mathrm{co}(\m U))\r) < 0
\] 
where $\mathrm{dist}(x,A)$ is the distance from a point $x$ to a set $A$ and $\mathrm{bd}(\mathrm{co}(\m U))$ is the boundary of $\mathrm{co}(\m U)$. 

\textbf{(iii)} We only need to consider the case $d\geq 3$, since the case $d=2$ follows from part \textbf{(ii)} of the lemma. It has been shown above that for $\alpha=0$, 
\[
\mathcal{U}(0)=\{\wh{z}\} \text{ and } \min_{v\in\mathbb{S}^{d-1}}\max_{u\in\mathcal{U}(0)}v^{\top}u\geq 0.
\]
Therefore, it suffices to show that for any $v\in\mathbb{R}^{d-1}$ such that $\max_{u\in\mathcal{U}(0)}v^{\top}u= 0$, there exists $w\in\mathcal{U}(\alpha)$ satisfying $w^{\top}v>0$. Consider the point $\wh{z}+\frac{1}{2}\alpha \cdot v\in B(\wh{z},\alpha)$. As $\wh{z}$ is unique, we have 
\[
W(\wh{z})>W\l(\wh{z}+\frac{\alpha}{2}v\r).
\]
However, for any $w\in\mathbb{S}^{d-1}$ such that $w^{\top}v\leq 0$, we have 
\[
G(w)-\frac{1}{\sqrt{2\pi}} \l\langle w, \wh{z}+\frac{\alpha}{2}v\r\rangle
\geq G(w)-\frac{1}{\sqrt{2\pi}} \langle w, \wh{z}\rangle
\geq W(\wh{z}). 
\]
Thus there exists $w\in\mathbb{R}^{d-1}$ satisfying $w^{\top}v>0$ such that 
\[
\l\{G(w)-\frac{1}{\sqrt{2\pi}} \langle w, \wh{z}+\frac{\alpha}{2}v\rangle\r\}<W(\wh{z}).
\]
We conclude that $w\in \mathcal{U}(\alpha)$ and $w^{\top}v>0$. 
To prove the second claim of part \textbf{(iii)}, it suffices to apply the preceding argument to $v=\frac{z-\wh{z}}{\|z-\wh{z}\|}$.
\end{proof}

\end{document}